\documentclass[12pt,a4paper,leqno]{article}
\usepackage[utf8x]{inputenc}
\usepackage[francais]{babel}
\usepackage[T1]{fontenc}
\usepackage{lmodern}
\usepackage{marvosym}
\usepackage{amsmath}
\usepackage{stmaryrd}
\usepackage{bbm}
\usepackage{oldgerm}
\usepackage{amssymb, amscd, amsthm, mathrsfs}
\usepackage[all]{xy}
\usepackage[dvips]{graphicx}
\usepackage{makeidx}

\makeindex
\newtheorem{theo}{{Th\'eor\`eme}}[section]
\newtheorem{coro}[theo]{{Corollaire}}
\newtheorem{lemma}[theo]{{Lemme}}
\newtheorem{prop}[theo]{Proposition}

\theoremstyle{remark}
\newtheorem{remark}[theo]{\textbf{Remarque}}

\theoremstyle{definition}
\newtheorem{defn}[theo]{D\'efinition}

\newcommand{\ra}{\rightarrow}

\newcommand{\ol}{\overline}
\newcommand{\immouv}[1][r]
   {\ar@{}[#1] |*[o][F]{\hbox{%
         \vrule width 1.5mm height 0pt depth 0pt%
         \vrule width 0pt height .75mm depth .75mm%
         }}
         \ar@{^{(}->}[#1]}

\newcommand{\cA}{\mathcal{A}}

\newcommand{\cC}{\mathcal{C}}

\newcommand{\cE}{\mathcal{E}}
\newcommand{\cF}{\mathcal{F}}
\newcommand{\cG}{\mathcal{G}}
\newcommand{\cH}{\mathcal{H}}

\newcommand{\cM}{\mathcal{M}}

\newcommand{\cO}{\mathcal{O}}
\newcommand{\cP}{\mathcal{P}}
\newcommand{\cR}{\mathcal{R}}

\newcommand{\A}{\mathbb A}
\newcommand{\C}{\mathbb C}
\newcommand{\B}{\mathbb B}
\newcommand{\F}{\mathbb F}

\newcommand{\I}{\mathbb I}

\newcommand{\N}{\mathbb N}

\newcommand{\Q}{\mathbb Q}

\newcommand{\T}{\mathbb T}

\newcommand{\Z}{\mathbb Z}

\newcommand{\D}{\mathbf{D}}

\newcommand{\bM}{\mathbf M}

\newcommand{\fC}{\mathfrak C}
\newcommand{\fD}{\mathfrak D}

\newcommand{\rH}{\mathrm H}

\newcommand{\m}{\mathrm m}

\newcommand{\sW}{\mathscr W}
\newcommand{\sV}{\mathscr V}
\newcommand{\sU}{\mathscr U}
\newcommand{\sX}{\mathscr X}

\DeclareMathOperator{\im}{\mathrm Im}

\newcommand{\Ev}{\mathbf Ev}
\newcommand{\Dir}{\mathbf{Dir}}

\DeclareMathOperator{\Hom}{\mathrm Hom}

\DeclareMathOperator{\cor}{\mathrm{cor}}

\DeclareMathOperator{\GL}{\mathrm GL}
\DeclareMathOperator{\SL}{\mathrm SL}
\DeclareMathOperator{\LC}{\mathrm LC}

\DeclareMathOperator{\LA}{\mathrm LA}

\DeclareMathOperator{\End}{\mathrm End}

\DeclareMathOperator{\Fil}{\mathrm{Fil}}
\DeclareMathOperator{\Frob}{\mathrm Frob}

\DeclareMathOperator{\Gal}{\mathrm Gal}

\DeclareMathOperator{\rank}{\mathrm rank}

\DeclareMathOperator{\Spm}{\mathrm{Spm}}

\DeclareMathOperator{\Sym}{\mathrm{Sym}}
\DeclareMathOperator{\tr}{\mathrm Tr}

\DeclareMathOperator{\plim}{\varprojlim}

\newcommand{\alg}{\mathrm{alg}}

\newcommand{\con}{\mathrm{cong}}
\newcommand{\cont}{\mathrm{cont}}
\newcommand{\cycl}{\mathrm{cycl}}
\newcommand{\dR}{\mathrm{dR}}

\newcommand{\cris}{\mathrm{cris}}

\newcommand{\rig}{\mathrm{rig}}

\newcommand{\Iw}{\mathrm{Iw}}
\newcommand{\Exp}{\mathrm{Exp}}

\newcommand{\z}{\zeta}
\newcommand{\G}{\Gamma}
\newcommand{\eps}{\epsilon}

\DeclareMathOperator{\Kato}{\mathrm{Kato}}

\DeclareMathOperator{\univ}{\mathrm{uinv}}

\begin{document}
\title{Le système d'Euler de Kato en famille (II)\footnote{2010 Mathematics Subject Classification. 11F85, 11F67, 11G40, 11R33, 11S80, 14G10, 14G35}}
\author{ Shanwen \textsc{WANG}
\footnotemark}
\date{}
\maketitle {\renewcommand{\thefootnote}{\fnsymbol{footnote}}
\footnotetext[1]{Email
: \textsf{wetiron1984@gmail.com}}}%

\def\abstractname{Résumé}

\begin{abstract}
Ce texte est le deuxième article sur une généralisation de système d'Euler de Kato. Il est consacré à la construction d'une famille de systèmes d'Euler de Kato sur la courbe de Hecke, qui interpole les systèmes d'Euler de Kato associés aux formes modulaires paramétrées par la courbe de Hecke cuspidale.
Par ailleurs, on explique la construction d'une famille de distributions sur $\Z_p$ sur la courbe de Hecke cuspidale à partir de cette famille de systèmes d'Euler de Kato; cette distribution fournit une fonction L $p$-adique en $2$-variable qui interpole les fonctions L $p$-adiques des formes modulaires précédentes.
\end{abstract}
\renewcommand{\abstractname}{Abstract}
\begin{abstract}
This article is the second article on the generalization of Kato's Euler system. The main subject of this article is to construct a family of Kato's Euler systems over the cuspidal eigencurve, which interpolate the Kato's Euler systems associated to the modular forms parametrized by the cuspidal eigencurve. We also explain how to use this family of Kato's Euler system to construct a family of distributions on $\Z_p$ over the cuspidal eigencurve; this distribution gives us a two variable $p$-adic L function which interpolate the $p$-adic L function of modular forms. 
\end{abstract}

\tableofcontents
\section{Introduction }
\subsection{Introduction}
Dans une séries d'articles datant des années $80$, Hida montre que les formes modulaires ordinaires vivent dans des familles $p$-adiques et le poids varie $p$-adiquement. En $1995$, Coleman montre que la même chose est vraie pour les formes modulaires surconvergentes non-ordinaires de pente finie. Ensuite, Coleman et Mazur \cite{CM} construisent un objet géométrique $\fC$, appelé la courbe de Hecke ("Eigencurve"), paramétrant les formes modulaires surconvergentes de pente finie. 
De plus, ils ont aussi construit une famille de représentations galoisiennes de rang $2$ sur la courbe de Hecke. On note $\fC^0$ la sous-courbe fermée de $\fC$, appelé la courbe de Hecke cuspidale, paramétrant les formes modulaires surconvergentes cuspidales de pente finie, ainsi que $\tilde{\fC}^0$ la normalisation de $\fC^0$.
Notre résultat principal (cf. théorème \ref{principal} ci-dessous) est que la fonction L $p$-adique d'une forme modulaire $f$ varie analytiquement avec $f$ sur $\tilde{\fC}^0$:

On choisit un caractère de Dirichlet $\chi$ modulo $N$ avec $N$ suffisantment grand\footnote{C'est une condition technique (cf. \S5.2.2 pour plus de détails) pour que l'on peut fixer les périodes en utilisant seulement un caractère $\chi$. } et $(N,p)=1$, ce qui permet de fixer les périodes par lesquelles on doit diviser les valeurs spéciales des fonctions L  que l'on veut interpoler (cf. \S \ref{cal}  pour détails). Si $f\in \fC^0$ est une forme propre classique non-critique de niveau modéré $\Gamma_1(N)$, on dispose d'une distribution $\mu_{f,\chi}$ sur $\Z_p^*$ à valeurs dans $\ol{\Q}_p$, telle que, quels que soient $0\leq j\leq k-2 $ et $\eta$ un caractère de Dirichlet modulo $p^m$ vérifiant que $\eta\chi(-1)=(-1)^{k-j-1}$, on a $\int_{\Z_p^*}\eta(x) x^{j}\mu_{f,\chi}= L(f\otimes\eta, j+1)$ à multiplication près par des facteurs explicites (facteurs d'Euler, périodes,$\cdots$), et qui fournit la fonction L $p$-adique attachée à $f$, en posant\footnote{$\langle\cdot\rangle$ est l'application de projection $\Z_p^*\ra 1+p\Z_p$.} 
\[L_{p,\chi}(f,\kappa, s)=\int_{\Z_p^*}\kappa(x)\cdot \langle x\rangle^{s}\cdot \mu_{f,\chi},\] si $\kappa$ est un caractère localement analytique de $\Z_p^*$ et $s\in \Z_p$.
\begin{theo}\label{principal}Si $x$ est un point classique non-critique de $\tilde{\fC}^0$, alors il existe un ouvert affinoïde $X\subset \tilde{\fC}^0$ contenant $x$ et une distribution $\mu_{X,\chi} $ sur $\Z_p^*$ à valeurs dans $\cO(X)$, tels que, pour tout point $f$ dans l'intersection de $X$ et le sous-ensemble $Z$ des formes propres classiques non-critiques de $\tilde{\fC}^0$, on a \[\mathrm{Ev}_f(\mu_{X,\chi})= C(f)\mu_{f,\chi},\] où $C(f)$ est une constante dans $\bar{\Q}_p^*$ dépendant de la forme $f$.
\end{theo}   


\begin{remark}(1) Il y a au moins trois manières de construire $\mu_{f,\chi}$ correspondant aux différentes réalisations des motifs associés aux formes modulaires: 

$\bullet$ La méthode classique utilise la réalisation de Betti, c'est à dire, la théorie de symboles modulaires (Mazur-Swinnerton-Dyer \cite{MS}, Manin \cite{Ma}, Vishik \cite{Vi}, Amice-Vélu \cite{Av}, Mazur-Tate-Teiltelbaum \cite{MTT}, Stevens \cite{St}, Pollack-Stevens \cite{SP} et \cite{SP1}...); 

$\bullet$ Une méthode plus récente, correspondante à la réalisation de de Rham, passe par la méthode de Rankin-Selberg (Hida \cite{HD} et \cite{HD1} dans le cas ordinaire, Panchishkin \cite{AP} dans le cas général); 

$\bullet$ La méthode de Kato \cite{KK}, via la réalisation étale $p$-adique, passe par la construction d'un système d'Euler, utilise la théorie des $(\varphi,\Gamma)$-modules de Fontaine \cite{Fo} pour en déduire, via une variante de l'application exponentielle de Perrin-Riou \cite{PR} et \cite{PR1}, une distribution. Montrer que cette distribution est celle que l'on cherche (i.e. interpole les valeurs spéciales de la fonction L complexe de la forme modulaire) nécessite de comparer deux lois de réciprocités explicites, et d'utiliser la méthode de Rankin comme dans l'approche de Panchishkin. 


(2) Des fonctions L $p$-adiques en deux variables, dont une variable varie sur un morceau de $\tilde{\fC}^0$, ont déjà été construites par des méthodes différentes correspondant aux constructions de $\mu_{f,\chi}$ ci-dessus:

$\bullet$ La stratégie de Stevens \cite{St1} (travail non publié), Pollack-Stevens \cite{SP1} et Bellaïche \cite{Be} est d'utiliser la théorie des symboles modulaires surconvergents, et ils réussissent à construire une fonction L $p$-adique $L_p(x,s)$ en deux variables, où $x$ varie dans un voisinage d'une forme modulaire raffinée (non-critique pour Stevens, et critique pour Bellaïche) sur la courbe de Hecke;

$\bullet$ La stratégie de Hida \cite{HD1} (pour la famille ordinaire) et de Panchishkin \cite{AP} (pour la famille de pente finie fixée) est d'utiliser la méthode de Rankin-Selberg en famille; 

$\bullet$ La stratégie d'Emerton \cite{EM} est d'utiliser la cohomologie complété et le foncteur de Jacquet dans la théorie de représentations localement analytiques de $\GL_2(\Q_p)$;

$\bullet$ La stratégie de Fukaya \cite{Fu} et de Delbourgo \cite{DD}, dans le cas ordinaire, passe par la déformation de systèmes d'Euler de Kato via la $K$-théorie et via la théorie des symboles modulaires respectivement, utilise la série de Coleman pour $K_2$ et une grande exponentielle duale respectivement pour en déduire une fonction L $p$-adique sur la famille ordinaire. 

(3) Notre stratégie a pour point de départ les travaux de Kato \cite{KK} et de Colmez \cite{PC1} (revisités par l'auteur dans $\cite{Wang}$) sur le système d'Euler de Kato.

$\bullet$ Dans \cite{WangI}, pour $c,d\in \Z_p^*$, on a construit une déformation $z_{\Kato,c,d}(\nu_j)$ (cf. \S 3.3) du système d'Euler de Kato sur l'espace des poids en reprenant la construction de Kato et défini une famille d'applications exponentielles duales, qui interpole l'application exponentielle duale de Kato et qui envoie la famille de systèmes d'Euler de Kato sur le produit d'une famille de séries d'Eisenstein avec une série d'Eisenstein.

$\bullet$ Dans cet article, on construit une famille de systèmes d'Euler de Kato sur $\fC^0$ (cf. \S 4) à partir de $z_{\Kato,c,d}(\nu_j)$; ensuite on utilise la théorie des $(\varphi,\Gamma)$-modules en famille (\cite{BC}, \cite{KJX}, \cite{Liu}) pour en déduire une distribution $\mu_{X,c,d,\chi}$ sur $\Z_p$ à valeurs dans $\cO(X)$ (cf. proposition \ref{principal_vrai}), où $X\subset\tilde{\fC}^0$ un ouvert affinoïde comme dans le théorème. En divisant la restriction de $\mu_{X,c,d,\chi}$ à $\Z_p^*$ par un facteur explicit (cf. la formule (\ref{CD})), on obtient la distribution voulue, qui est indépendante du choix de $c,d$.

(4) Le théorème ci-dessus montre qu'il existe une fonction L $p$-adique en deux variables $L_p(x,s)$, où $x$ varie sur $\tilde{\fC}^0$, interpolant les fonctions L $p$-adiques de formes modulaires. Par prolongement analytique, on en déduit que le théorème \ref{principal} est encore valable aux points classiques cuspidals critiques.


\end{remark}

Le plan de cet article est le suivant: la démonstration comporte deux étapes principales mentionnées dans la remarque ci-dessus, qui correspondent aux chapitres $\S 4$ et $\S 5$. Ces deux étapes reposent sur deux chapitres de préparations (\S 2 et 3):  au chapitre \S 2, on rappelle la théorie des $(\varphi,\Gamma)$-modules et la triangulation en famille; au chapitre \S 3, on rappelle la construction de la famille de système d'Euler de Kato sur l'espace des poids et ses variantes, à qui on appliquera la projection du système d'Euler de Kato sur $\fC^0$ (cf. \S 4).\\

 \subsection*{Remerciements:} Ce travail repose sur les travaux d'Ash-Stevens, Berger-Colmez, Bellaïche, Chenevier, Colmez, Kato, Kedlaya-Pottharst-Xiao, et Liu. Je tiens à leur exprimer ma gratitude. Pendant la préparation de cet article, j'ai bénéficié de communications et discussions avec F. Andreatta, J. Bellaïche, D. Benois, P. Colmez, A. Iovita, R. Liu, G. Stevens, J. Tong et L. Xiao.  Je voudrais aussi remercier les Prof. Y. Tian et Prof. S. Zhang et le Morningside Center de Pékin, ainsi que les Prof. H. Chen et Prof. L. Fu et le CIM de Tianjin,  pour leur hospitalité; les exposés que j'ai donnés lors des conférences à ces deux endroits, en août 2012 et juin 2013, m'ont grandement aidé à mettre mes idées au claire.  Je remercie aussi le Cariparo Eccellenza Grant, le projet SFB 45 et le projet SFB 1085,  d'avoir financé mes séjours de 2011 à 2013 à Padoue, Italie, et de 2014 à Essen et à Regensburg, Allemagne. Les rédactions de cet article a été faite pendant mes séjours à l'IMJ et à l'IHES en 2013, et au CRM de montréal en 2015. Je souhaite remercier ces institutions pour m'avoir fourni d'excellentes conditions de travail.

\subsection{Notations}
On note $\overline\Q$ la cl\^oture alg\'ebrique de $\Q$ dans $\C$,
et \index{on} fixe, pour tout nombre premier~$p$, une cl\^oture
alg\'ebrique $\overline\Q_p$ de $\Q_p$, ainsi qu'un plongement de
$\overline\Q$ dans $\overline\Q_p$.

Si $N\in\N$, on note $\zeta_N$ la racine $N$-i\`eme
$e^{2i\pi/N}\in\overline\Q$ de l'unit\'e, et on note
$\Q^{\rm cycl}$ l'extension cyclotomique de $\Q$,
r\'eunion des $\Q(\zeta_N)$, pour $N\geq 1$, ainsi que $\Q^{\rm cycl}_p$ l'extension cyclotomique de $\Q_p$, r\'eunion de $\Q_p(\z_N)$, pour $N\geq 1$.
 \subsubsection*{Objets ad\'eliques}
 Soient $\cP$ l'ensemble des nombres premiers de $\Z$ et $\hat{\Z}$ le compl\'et\'e profini de $\Z$, alors
\index{qi si wo le} 
$\hat{\Z}=\prod_{p\in\cP}\Z_p$. Soit $\A_f=\Q\otimes\hat{\Z}$ l'anneau
des ad\`eles finis de $\Q$. Si
$x\in\A_f$, on note $x_p$ (resp. $x^{]p[}$) la
composante de $x$ en $p$ (resp. en dehors de $p$). Notons
$\hat{\Z}^{]p[}=\prod_{l\neq p}\Z_l$. On a donc
$\hat{\Z}=\Z_p\times\hat{\Z}^{]p[}$. Cela induit les d\'ecompositions
suivantes: pour tout  $d\geq 1$,
\[
\bM_d(\A_f)=\bM_d(\Q_p)\times\bM_d(\Q\otimes\hat{\Z}^{]p[})
\text{ et }
\GL_d(\A_f)=\GL_d(\Q_p)\times\GL_d(\Q\otimes\hat{\Z}^{]p[}).\]
On d\'efinit les sous-ensembles suivants de $\A_f$ et
$\bM_2(\A_f)$:
\begin{align*}\hat{\Z}^{(p)}=\Z_p^{*}\times\hat{\Z}^{]p[} &\text{ et
}
\bM_{2}(\hat{\Z})^{(p)}=\GL_2(\Z_p)\times\bM_2(\hat{\Z}^{]p[}), \\
\A_f^{(p)}=\Z_p^{*}\times(\Q\otimes\hat{\Z}^{]p[})
&\text{ et }
\bM_{2}(\A_f)^{(p)}=\GL_2(\Z_p)\times\bM_2(\Q\otimes\hat{\Z}^{]p[}).
\end{align*}
\subsubsection*{Actions de groupes}
Soient $X$ un espace topologique localement profini, et $V$ un
$\Z$-module. On note $\LC_c(X,V)$ le module des fonctions localement
constantes sur  $X$ \`a valeurs dans $V$ dont le support est compact
dans $X$. On note $\fD_{\alg}(X,V)$ l'ensemble des distributions alg\'ebriques sur
$X$ \`a valeurs dans $V$, c'est \`a dire des applications
$\Z$-lin\'eaires de $\LC_c(X,\Z)$ \`a valeurs dans $V$. On note $\int_X\phi\mu$ la valeur de $\mu$
sur $\phi$ o\`u $\mu\in\fD_{\alg}(X,V)$ et $\phi\in \LC_c(X,\Z)$.

Soit $G$ un groupe localement profini, agissant contin\^ument \`a
droite sur $X$ et $V$. On munit $\LC_c(X,\Z)$ et
$\fD_{\alg}(X,V)$ d'actions de $G$ \`a droite comme suit:

si $g\in G, x\in X,\phi\in\LC_c(X,\Z), \mu\in\fD_{\alg}(X,V),$ alors
\begin{equation}\label{actiondis} (\phi*g)(x)=\phi(x*g^{-1}) \text{ et } \int_{X}\phi(\mu*g)=\bigl(\int_{X}(\phi*g^{-1})\mu\bigr)*g.
\end{equation}
Si $M$ est un $G$-module topologique à droite, on note $\rH^i(G,M)$ le $i$-ième groupe de cohomologie continue de $G$ à valeurs dans $M$. Si $X$ est en plus muni d'une action à gauche de $G$ (notée $(g,x)\mapsto g\star x$) commutant à l'action à droite de $G$, les modules $\rH^i(G, \fD_{\alg}(X,M))$ sont naturellement des $G$-modules à gauche. 
\subsubsection*{Formes modulaires}
Soient $A$ un sous-anneau de $\C$ et $\Gamma$ un sous-groupe
d'indice fini de $\SL_2(\Z)$. On note $\cM_k(\Gamma,\C)$ le
$\C$-espace vectoriel des formes modulaires de poids $k$ pour
$\Gamma$. On note aussi $\cM_{k}(\Gamma,A)$ le sous $A$-module de
$\cM_k(\Gamma,\C)$ des formes modulaires dont le $q$-d\'eveloppement
est \`a coefficients dans $A$. On pose
$\cM(\Gamma,A)=\oplus_{k=0}^{+\infty}\cM_k(\Gamma,A)$. Et on note
$\cM_k(A)$ (resp. $\cM(A)$) la r\'eunion des $\cM_k(\Gamma,A)$
(resp. $\cM(\Gamma,A)$), o\`u $\Gamma$ d\'ecrit tous les
sous-groupes d'indice fini de $\SL_2(\Z)$. 

On d\'efinit de m\^eme:
\[\cM^{\con}_k(A)=\bigcup\limits_{\substack{\Gamma \text { sous-groupe de congruence } }}\cM_k(\Gamma,A)\text{ et }
\cM^{\con}(A)=\bigcup_k\cM_k^{\con}(A).\]

Soit $K$ un sous-corps de $\C$ et soit $\ol{K}$ la cl\^oture alg\'ebrique de $K$. On note $\Pi_K$ le groupe des
automorphismes de $K$-algèbres graduées $\cM(\bar{K})$ sur $\cM(\SL_2(\Z),K)$; c'est un
groupe profini.
Si $f\in\cM(\ol{K})$, le groupe de galois $\cG_K$ agit sur les coefficients du $q$-d\'eveloppement de $f$; ceci nous fournit une section de $\Pi_K\ra \cG_K$, notée par $\iota_K$. 

Le groupe des automorphismes de $\cM^{\con}(\Q^{\cycl})$ sur $\cM(\SL_2(\Z),\Q^{\cycl})$ est le groupe $\SL_2(\hat{\Z})$, le compl\'et\'e profini de $\SL_2(\Z)$ par rapport aux sous-groupes de congruence. D'autre part, soit $f\in\cM^{\con}(\Q^{\cycl})$, le groupe $\cG_{\Q}$ agit sur les coefficients du $q$-d\'eveloppement de $f$ \`a travers son quotient $\Gal(\Q^{\cycl}/\Q)$ qui est isomorphe \`a $\hat{\Z}^{*}$ par le caract\`ere cyclotomique $\chi_{\cycl}$. On note $H$ le groupe des automorphismes de $\cM^{\con}(\Q^{\cycl})$ sur $\cM(\SL_2(\Z),\Q)$.  La sous-alg\`ebre $\cM^{\con}(\Q^{\cycl})$ est stable par $\Pi_{\Q}$ qui agit \`a travers $H$.
Le groupe $H$ est isomorphe à $\GL_2(\hat{\Z})$ et on a le diagramme commutatif de groupes suivant (cf. par exemple \cite[théorème 2.2]{Wang}):
\begin{equation}\label{diagram} 
\xymatrix{
1\ar[r]&\Pi_{\bar{\Q}}\ar[r]\ar[d]&\Pi_{\Q}\ar[r]\ar[d]&\cG_{\Q}\ar[r]\ar[d]^{\chi_\cycl}\ar@{.>}@/^/[l]^{\iota_\Q}&1\\
1\ar[r]&\SL_{2}(\hat{\Z})\ar[r]&\GL_2(\hat{\Z})\ar[r]^{\det}&\hat{\Z}^{*}\ar[r]\ar@{.>}@/^/[l]^{\iota}&1                     },
\end{equation}
o\`u la section $\iota_\Q$ de $\cG_{\Q}$ dans $\Pi_{\Q}$ d\'ecrite plus haut envoie $u\in \hat{\Z}^{*}$ sur la matrice $(\begin{smallmatrix}1&0\\0&u\end{smallmatrix})\in \GL_2(\hat{\Z})$.

\subsubsection*{Anneaux de séries de Laurent}

Fixons une extension finie $L$ de $\Q_p$. Le caractère cyclotomique $\chi_{\cycl}$ induit un isomorphisme de $\Gamma=\Gal(\Q_p(\z_{p^\infty})/\Q_p)$ sur $\Z_p^*$. Soient $\cR^{+}$ l'anneau des fonctions analytique sur le disque $v_p(T)>0$ à coefficient dans $L$, $\cE^+$ le sous-anneau de $\cR^+$ des éléments bornés, $\cR$ l'anneau des fonctions annalytiques sur une couronne  $0<v_p(T)\leq r$, o\`u $r>0$ d\'epend de l'\'el\'ement consid\'er\'e (l'anneau de Robba), $\cE^{\dag}$ le sous-anneau de $\cR$ des éléments bornés (c'est un corps) et $\cE$ le complété de $\cE^\dag$ pour la valuation $p$-adique. On munit ces anneaux d'actions continues de $\Gamma$ et d'un Frobenius $\varphi$, commutant entre elles, en posant $\varphi(T)=(1+T)^p-1$ et $\gamma(T)=(1+T)^{\chi_{\cycl}(\gamma)}-1$ si $\gamma\in \Gamma$.

Soit $C$ un pro-$p$-groupe qui est isomorphe à $1+p\Z_p$. Si $c$ est un générateur de $C$, l'algèbre de groupe complété $\Lambda_C$ de $C$ est isomorphe à $\Z_p[[c-1]]$. On définit l'anneau $\cR^+(C)$ en remplaçant par $c-1$ la variable $T$ intervenant dans la définition de $\cR^+$.
Si $C_n$ est le sous-groupe fermé de $C$ d'indice $p^n$, on a un isomorphisme $\Lambda_C\otimes_{\Lambda_{C_n}}\cR^+(C_n)\cong \cR^+(C)$. 

Soit $H$ un groupe isomorphe à $\Z_p^*$. On note $H_d$ le sous-groupe de $H$ correspondant à $1+p^d\Z_p$. On définit l'anneau $\cR^+(H)$ par le produit tensoriel
$\Lambda_H\otimes\cR^+(H_d)$, qui est indépendant du choix de $H_d$. 

\section{$(\varphi,\Gamma)$-modules et représentations galoisiennes}

\subsection{Raffinement et triangulation}
On dispose d'une équivalence de catégories (grâce à Fontaine, raffinée par Cherbonnier-Colmez \cite{CC}, Berger \cite{B} et Kedlaya \cite{Ke}) entres la catégories des $L$-représentations de $\cG_{\Q_p}$ et celle des $(\phi,\Gamma)$-modules étales sur $\cE$, $\cE^\dag$ et $\cR$ respectivement. Si $V$ est une $L$-représentation de $\cG_{\Q_p}$, on note $\D^\dag(V)$ et $\D_{\rig}(V)$ respectivement les $(\varphi,\Gamma)$-modules sur $\cE^\dag$ et $\cR$ associés à $V$.

\begin{defn} $(1)$ Un $(\varphi, \Gamma)$-module $D$ sur $\cR$ est dit triangulable si c'est une extension successive
de $(\varphi, \Gamma)$-modules de rang $1$ sur $\cR$, i.e. si $D$ possède une filtration croissante par des
sous-$(\varphi, \Gamma)$-modules $D_i$, pour $0 \leq i \leq d$, telle que l'on ait $D_0 = 0, D_d = D$ et $D_i/D_{i-1}$ est libre
de rang $1$ si $1\leq i \leq d$.\\
$(2)$ Soit $V$ une $L$-représentation de $\cG_{\Q_p}$.  On dit que $V$ est trianguline si $\D_{\rig}(V)$ est triangulable. 
\end{defn}

Soit $D= \D_{\rig}(V)$ le $(\varphi,\Gamma)$-module sur $\cR$ associé à une $L$-représentation $V$ de $\cG_{\Q_p}$. On pose $\D_{\cris}(V)=(D\otimes \cR[\frac{1}{t}])^{\Gamma}$, où $t=\log(1+T)$; c'est un $L$-espace vectoriel de dimension $\leq \rank_{\cR} D$ muni d'une action $L$-linéaire de $\varphi$ induite par celle sur $D$ et d'une filtration induite par celle sur $L_{\infty}((t))$ via l'application de localisation $\iota_\infty: \cR\ra L_{\infty}((t))$.  On dit que $V$ est une représentation cristalline si $\dim_L\D_{\cris}(V)=\rank_{\cR} D$. 
D'après Berger \cite{B}, cette définition coïncide avec la définition usuelle. 

Soit $V$ est une représentation cristalline de $\cG_{\Q_p}$ de dimension $2$. Si $\cF$ est un sous-espace propre de $\D_{\cris}(V)$ stable par $\varphi$, alors $D_{\cF}=\cR[\frac{1}{t}]\cF\cap \D_{\rig}(V)$ est un sous-$(\varphi,\Gamma)$-module de rang $1$ sur $\cR$, et $0\varsubsetneq D_\cF\varsubsetneq D$ est une triangulation de $D$. Ceci induit un cas particulier de\footnote{ La proposition \cite[\text{proposition} 2.4.1]{BeCh} décrit une telle bijection pour $V$ une représentation de dimension quelconque. } \cite[\text{proposition} 2.4.1]{BeCh}.
\begin{prop}\label{raff_BC}Il existe une bijection entre les sous-espaces propres stables par $\varphi$ de $\D_{\cris}(V)$ et les triangulations de $\D_{\rig}(V)$, dont l'inverse est donné par $\cF=(D_{\cF}[\frac{1}{t}])^{\Gamma}$. 
\end{prop}
Soient $N\geq 1$ un entier \emph{premier à} $p$, $k\geq 2$ un entier, et $\eps$ un caractère
de Dirichlet modulo N (pas nécessairement primitif). Fixons une forme modulaire primitive $f$ de niveau $\Gamma_1(N)$ de caractère de Dirichlet $\eps$. En particulier, $a_1=1$ et $\Q(f)=\Q(a_2, . . . , a_n, . . . )$ est une extension finie de $\Q$.

 Si $S$ est un sous-ensemble fini de $\cP$ tel que $\{l\in \cP: l|N\}\subset S$, et si $M$ est un $\Q(f)$-espace vectoriel muni d'actions des $T(l), T'(l), l\notin S\cup\{p\}$, l'opérateur $T_p$ et $(\begin{smallmatrix}u^{-1}& 0\\ 0& u\end{smallmatrix})$, pour $u\in U$ sous-groupe ouvert de $\hat{\Z}^*$, on note $M_{\pi_{f}}$ le quotient de $M$ par le sous-$\Q(f)$-espace vectoriel engendré par les $x*T(l)-a_lx$, où $x\in M$ et  $l\notin S\cup\{p\}$. 

D'après Deligne \cite{PD}, on sait associer une représentation galoisienne de $\cG_{\Q}$ de dimension $2$ à $f$. De manière explicite:  on note $\bar{\Gamma}_1(N)$ le complété profini de $\Gamma_1(N)$ et on note $V_p= \Q_pe_1 \oplus \Q_pe_2$ la représentation standard
de dimension $2$ de $\GL_2(\Z_p)$ donnée par $e_1*\gamma = ae_1 + be_2$ et $e_2*\gamma= ce_1 + de_2$ si
$\gamma=(\begin{smallmatrix}a &b\\
c& d\end{smallmatrix})\in\GL_2(\Z_p)$.
On définit la représentation galoisienne de $\cG_\Q$ associée à $f$ par 
\[ V_{f}=(\rH^1(\bar{\Gamma}_1(N), \Sym^{k-2}V_p)\otimes_{\Q_p}\Q_p(f))_{\pi_{f}}\otimes_{\Q_p}\Q_p(2-k).\] 
C'est une $\Q_p(f)$-représentation irréductible de $\cG_{\Q}$ de dimension $2$, non ramifiée en dehors de $Np$ . Si $l\nmid Np$, le déterminant de $1-\Frob_l^{-1}X$ agissant sur $V_{f}$ est $1-a_lX+\eps(l)l^{k-1}X^2$, où $\Frob_l$ est un frobenius arithmétique en $l$. La restriction de $V_f$ à $\cG_{\Q_p}$ est cristalline de poids de Hodge-Tate\footnote{ Par convention, le poids de Hodge-Tate du caractère cyclotomique est $1$.} $0,1-k$ et le polynôme caractéristique de $\varphi$ sur $\D_{\cris}(V_f)$ est $X^2-a_pX+\eps(p)p^{k-1}$.

On note $\Gamma(N;p)=\Gamma_1(N)\cap \Gamma_0(p)$. En tant que représentation de $\cG_{\Q}$, on a \[\rH^1(\Gamma(N;p), \Sym^{k-2}V_p(1))_{\pi_f}\cong V_{f}(k-1)\otimes V_f(k-1).\]
\begin{defn}Si $\alpha$ est une racine du polynôme caractéristique de $T_p$, on note $V_{f_\alpha}$ le plus grand espace quotient de $\rH^1(\Gamma(N;p), \Sym^{k-2}V_p(1))_{\pi_f}$ propre pour l'opérateur $U_p$ avec la valeur propre $\alpha$; on dit que $V_{f_\alpha}$ est \emph{un raffinement de} $V_f$. En tant que représentation de $\cG_{\Q}$, on a $V_{f_\alpha}\cong V_{f}(k-1)$.
\end{defn}

 Par l'isomorphisme d'Eichler-Shimura $\cite{Fa}$, on a le diagramme commutatif Hecke-Galois équivariant de $\Q_p(f_\alpha)\otimes\C_p$-modules 
\[\xymatrix{ (\rH^1(\Gamma_1(N), \Sym^{k-2}V_p(1))\otimes\Q_p(f_\alpha))\otimes_{\Q_p}\C_p \ar[r]\ar[d]& M_k(\Gamma_1(N),\Q_p(f_\alpha))\otimes \C_p\ar[d] 
\\ 
(\rH^1(\Gamma(N;p), \Sym^{k-2}V_p(1))\otimes\Q_p(f_\alpha) )\otimes_{\Q_p}\C_p\ar[r]\ar[d] &M_k(\Gamma(N;p),\Q_p(f_\alpha))\otimes\C_p\ar[d]
\\ 
V_{f_\alpha }\otimes\C_p\ar[r]&\Q_p(f_\alpha) f_\alpha\otimes\C_p},\]   
où l'action de $\cG_{\Q}$ est linéaire sur $\Q_p(f_\alpha)$, semi-linéaire sur $\C_p$ à travers $\cG_{\Q_p}$. Le corollaire suivant est une conséquence immédiate de la proposition \ref{raff_BC}.

\begin{coro}Soit $f$ une forme modulaire primitive de niveau $\Gamma_1(N)$. Soit $V_f$ la représentation galoisienne associée à $f$ comme ci-dessus. Alors il existe une bijection naturelle entre l'ensemble des raffinements de $V_f$ et l'ensemble des triangulations de $D=\D_{\rig}(V_f)$.
\end{coro}
\subsection{Familles faiblement raffinées}
\subsubsection*{Le foncteur de Berger-Colmez $\D_{\mathrm{BC},\rig}$}
Soit $S$ une $\Q_p$-algèbre de Banach et soit $\cO_S$ l'anneau des éléments de $S$ de norme $\leq 1$. On note $M(S)$ l'espace rigide analytique associé à $S$. Si $x\in M(S)$, on note $m_x$ l'idéal maximal de $S$ correspondant à $x$. 
\begin{defn} Une $S$-représentation $V_S$ de $\cG_{\Q_p}$ est un $S$-module localement libre muni d'une action continue $S$-linéaire de $\cG_{\Q_p}$.

\end{defn}
Si $r>0$, on note $\cR_r$ l'anneau des fonctions analytiques sur la couronne 
\[I_r=\{T\in\C_p: 0< v_p(T)\leq \frac{1}{r}\}\] et on a $\cR=\cup_{r>0}\cR_r$. On note $\cR_{r,S}=S\hat{\otimes}\cR_r$ et $\cR_S=\cup_{r>0}\cR_{r,S}$.

Le foncteur $V\mapsto \D_{\rig}(V)$ s'étend aux $S$-représentations de $\cG_{\Q_p}$ grâce aux travaux de Berger et Colmez \cite{BC}, Kedlaya et Liu \cite{KL}; de manière précise, on dispose (voir \cite[\S 3]{KL}) d'un foncteur $\D_{\mathrm{BC},\rig}$ de la catégorie des $S$-représentations de $\cG_{\Q_p}$ sur la catégorie des $(\varphi,\Gamma)$-modules étales sur $\cR_S$, vérifiant, si $V_S$ est une $S$-représentation de rang $d$,  \\
$(1)$ le $(\varphi,\Gamma)$-module $\D_{\mathrm{BC},\rig}(V_S)$ est localement libre de rang $d$ sur $\cR_S$;\\
$(2)$ si $x\in \Spm S$, on a  $\D_{\mathrm{BC},\rig}(V_S)\otimes_S(S/m_x)\cong \D_{\rig}(V_x)$, où $V_x=V_S\otimes_S(S/m_x)$.

\subsubsection*{Familles faiblement raffinées}
Soit $\sX$ un espace rigide analytique réduit et séparé. 
\begin{defn}Une famille faiblement raffinée de représentations $p$-adiques de dimension $2$ sur $\sX$ est la donnée d'un $\cO(\sX)$-module $V_{\sX}$, localement libre de rang $2$, muni d'une action $\sX$-linéaire continue de $\cG_{\Q_p}$ et de données suivantes:

$(1)$ $2$ fonctions analytiques $\kappa_1,\kappa_2\in \cO(\sX)$,

$(2)$ une fonction analytique $F\in \cO(\sX)$,

$(3)$ un sous-ensemble Zariski-dense $Z$ de $\sX$,\\
qui vérifient les conditions suivantes:

$(a)$ Les poids de Hodge-Tate de $V_\sX$ sont $\kappa_1,\kappa_2$;

$(b)$ Si $z\in Z$, alors $V_z$ est cristalline et $\kappa_1(z)$ est le plus grand poids de Hodge-Tate de $V_z$;

$(c)$ Si $z\in Z$, $\D_\cris(V_z)$ admet d'un sous $\varphi$-module propre de rang $1$ avec la valeur propre 
$p^{-\kappa_1(z)}F(z)$; 


$(d)$ Si $C\in \N$, on note $Z_C$ l'ensemble des $z\in Z$ tels que $\kappa_1(z)-\kappa_2(z)>C$. Alors pour tout $C$, $Z_C$ s'accumule en tout les points de $Z$. Autrement dit, pour tous $z\in Z$, $C>0$, et tout voisinage affinoïde $U$ de $z$, $U\cap Z_C$ est Zariski dense dans $U$;

$(e)$ Il existe un caractère continu $\eta: \Z_p^*\ra \cO(\sX)^*$, dont la dérivée en $1$ est l'application $\kappa_1$ et l'évaluation en $z\in Z$ est le caractère $x\mapsto x^{\kappa_1(z)} $.
\end{defn}
\begin{remark}La notion de famille faiblement raffinée en dimension quelconque a été étudiée par Bellaïche et Chenevier $\cite[\S 4.2]{BeCh}$.\end{remark}
La proposition suivant, qui dit qu'une famille faiblement raffinée de représentations $p$-adiques de dimension $2$ admet d'une triangulation globale, est une conséquence facile de \cite[\text{theorem} 0.3.4]{Liu}.  
\begin{prop}\label{triangulation} Soit $(V_\sX, Z, F, \kappa_1,\kappa_2,\eta)$ une famille faiblement raffinée de représentations $p$-adiques de dimension $2$. Alors $\D_{\mathrm{BC},\rig}(V_\sX(\eta^{-1}))^{\varphi=F,\Gamma=1}\otimes_{\cO(\sX)}\D_{\mathrm{BC},\rig}(\eta)$ est un faisceau cohérent sur $\sX$ de rang $1$. En particulier, on a une triangulation en famille de $\D_{\mathrm{BC},\rig}(V_\sX)$:
\[0\subset \D_{\mathrm{BC},\rig}(V_\sX(\eta^{-1}))^{\varphi=F,\Gamma=1}\otimes_{\cO_\sX} \D_{\mathrm{BC},\rig}(\eta) \subset \D_{\mathrm{BC},\rig}(V_\sX).\]
\end{prop}

\section{Famille de Système d'Euler de Kato sur l'espace des poids et ses variantes}

\subsection{L'espace des poids $\sW$ et le caractère universel}
On note $\Lambda=\Z_p[[\Z_p^*]]$ l'algèbre d'Iwasawa. Le groupe rigide analytique $\sW$ sur $\Q_p$ qui lui est associé est appelé \emph{l'espace des poids}.  Les $\C_p$-points de $\sW$ constituent l'ensemble $\Hom_{\cont}(\Z_p^*,\C^*_p)$ des caractères continus sur $\Z_p^*$ à valeurs dans $\C_p$. La décomposition de $\Z_p^*\cong\mu_{p-1}\times (1+p\Z_p)$, où $\mu_{p-1}$ est le groupe des unités d'ordre $p-1$, induit une décomposition 
\[\sW(\C_p)=\Hom(\mu_{p-1},\C_p^*)\times \Hom_{\cont}(\Gamma,\C_p^*), \text{ où } \Gamma=1+p\Z_p \text{ est un pro-$p$-groupe}.\]  On a une inclusion $\Z\subset\sW$ envoyant $k$ sur le caractère $(z\mapsto z^{k-2})$.
 
L'inclusion naturelle $\Z_p^*\subset\Z_p[\Z_p^*]$ induit un caractère canonique \[\kappa^{\univ}:\Z_p^*\ra \Z_p[\Z_p^*] \subset \cO(\sW),\] appelé le \emph{caractère universel} de $\sW$. Soit $U$ un ouvert affinoïde de $\sW$. Si $\kappa\in U(\C_p)$, on définit une application $\mathrm{Ev}_\kappa: \cO(U)\ra\C_p$ d'évaluation en $\kappa$. En particulier, on a $\kappa=\mathrm{Ev}_\kappa\circ\kappa^{\univ}$.

On rappelle que si $\mathscr{X}$ est un $\Q_p$-espace rigide, un sous-ensemble $Z\subset \mathscr{X}$ est dit \emph{Zariski-dense} si pour tout sous-ensemble analytique $U\subset \mathscr{X}$ tel que $Z\subset U$, on a alors $U=\mathscr{X}$.  Soit $Z\subset \mathscr{X}$ un sous-ensemble Zariski-dense, tel que pour tout $z\in Z$ et tout voisinage ouvert affinoïde $V$ de $z$ dans $\mathscr{X}$, $V\cap Z$ est Zariski-dense dans chaque composante irréductible de $V$ contenant $z$, on dira alors que $Z$ est \emph{très Zariski-dense} dans $\mathscr{X}$. Un exemple important est que l'ensemble $\N$ (resp. $\Z$) est très Zariski-dense dans l'espace $\sW$.

On dira qu'un ouvert affinoïde $U=\Spm R\subset \sX$ est \emph{agréable} si l'anneau $R$ et son anneau résiduel $\tilde{R}=R^0/pR^0$ sont des anneaux principaux, où $R^0$ est l'anneau des entiers de $R$. Pour tout $1\leq n\in \N$ et $k\in\Z\subset \sW$,  la boule fermée $W_{k,n}$ de centre $k$ de rayon $p^{-n}$ est un ouvert affinoïde agréable de $\sW$. En particulier, tout point $k\in\Z$ admet une base de voisinages d'ouverts affinoïdes agréables dans $\sW$.

\subsection{Familles de représentations de $\I_0(p)$ sur $\sW$}\label{system}
On note $\I_0(p)$ le sous-groupe d'Iwahori de $\GL_2(\Z_p)$. Le but de ce paragraphe est de construire, en modifiant la construction de la représentation $\mathbf{Mes}_{\sW,j}$ de $\I_0(p)$ dans \cite[\S 2.2]{WangI}, des grosses représentations de $\I_0(p)$, qui sont des variantes de $\mathbf{Mes}_{\sW,j}$ et apparaissent dans la construction de la famille de systèmes d'Euler de Kato sur $\fC^0$.

\subsubsection*{$(1)$ La $\cO^{\mathrm{bd}}(\sW)$-représentation $\mathbf{Mes}^{\mathrm{bd}}_{\sW,j}$} 
On note $\cO^{\mathrm{bd}}(\sW)$ l'anneau des fonctions analytiques bornées sur l'espace des poids, qui est isomorphe à l'anneau $\Z_p[[\Z_p^*]]\otimes\Q_p$.   On définit une action $\cO^{\mathrm{bd}}(\sW)$-linéaire à gauche de $\I_0(p)$ sur le $\cO^{\mathrm{bd}}(\sW)$-module $\cC^0(\Z_p,\cO^{\mathrm{bd}}(\sW))$ des fonctions continues sur $\Z_p$ à valeurs dans $\cO^{\mathrm{bd}}(\sW)$ par la formule:
\[\gamma f(z)=f(\frac{b+dz}{a+cz}), \text{ si } f(z)\in \cC^0(\Z_p,\cO^{\mathrm{bd}}(\sW)) \text{ et } \gamma=(\begin{smallmatrix}a&b\\ c& d\end{smallmatrix})\in\I_0(p).  \]
Soit $1\leq j\in \N$. Soient $a\in\Z_p^*$ et $c\in p\Z_p$. La fonction
 \[\rho_j^{\mathrm{univ}}(\gamma)=\kappa^{\mathrm{univ}}(a+cz)\det\gamma^{-j}, \text{ où } \gamma=(\begin{smallmatrix}a&b\\ c& d\end{smallmatrix})\in\I_0(p),\]
  est un $1$-cocycle sur $\I_0(p)$ à valeurs dans le groupe des unités de l'anneau $\cC^0(\Z_p,\cO^{\mathrm{bd}}(\sW))$.

On note $\cC^{0,\mathrm{bd}}_{\sW,j}$ la $\cO^{\mathrm{bd}}(\sW)$-représentation de Banach de $\I_0(p)$, dont l'action de $\I_0(p)$ est donnée par la formule $\gamma f(z)=\rho_j^{\mathrm{univ}}(\gamma) f(\frac{b+dz}{a+cz})$.
Ceci nous permet d'appliquer la construction \cite[\S 2.2]{WangI} à la donnée
  $(\cC^0(\Z_p,\cO^{\mathrm{bd}}(\sW)),\rho^{\univ}_j)$; il en résulte une $\cO^{\mathrm{bd}}(\sW)$-représentation de Banach à droite  $\mathbf{Mes}_{\sW,j}^{\mathrm{bd}}$, qui est isomorphe à $\fD_0(\Z_p,\Q_p)\hat{\otimes}\cO^{\mathrm{bd}}(\sW)$ comme $\cO^{\mathrm{bd}}(\sW)$-modules. Elle donne une interpolation $p$-adique en poids $k$ des représentations algébriques $V_{k,j+2}$ de $\I_0(p)$, où $V_{k,j}=\Sym^{k-2}V_p\otimes \det^{2-j}$ avec $V_p$ la $L$-représentation standard de $\GL_2(\Z_p)$. Plus précisément, si $k\in\Z$, on note $\mathbf{Mes}^{\mathrm{bd}}_{k,j}=\mathrm{Ev}_k(\mathbf{Mes}^{\mathrm{bd}}_{\sW,j})$ et on définit une application $\Q_p$-linéaire continue $\I_0(p) $-équivariante
$\pi_{k,j}:\mathbf{Mes}^{\mathrm{bd}}_{k,j}\ra V_{k,j+2}$ par l'intégration: $\mu\mapsto \int_{\Z_p} f(z)\mu $, où  la fonction $f(z)=(e_1+ze_2)^{k-2}t^{-j}$ est à valeurs dans $V_{k,j+2}$. En composant l'application d'évaluation $\mathrm{Ev}_{k}$ et l'application $\pi_{k,j}$, on obtient une application de spécialisation
$\mathrm{Sp}_{k,j}: \mathbf{Mes}^{\mathrm{bd}}_{\sW,j}\ra V_{k,j+2}$,  qui est $\I_0(p)$-équivariante. En particulier, la masse de Dirac $\delta_0$ en $0$ fournit un élément $\nu_j$  de $\mathbf{Mes}_{\sW,j}^{\mathrm{bd}}$  qui interpole le vecteur de plus haut poids $e_1^{k-2}t^{-j}$ dans $V_{k,j+2}$ (i.e. on a $\mathrm{Sp}_{k,j} (\nu_j)= e_1^{k-2} t^{-j} $).

\subsubsection*{$(2)$ Les $W$-représentations $\D_{r,W,j}$ }
 
Si $r>0$, on note $\LA_{r}(\Z_p,L)$ l'espace des fonctions $f:\Z_p\ra L$ dont la restriction à $a+p^r\Z_p$ est la restriction d'une fonction $L$-analytique sur le disque fermé $\{x\in \C_p,v_p(x-a)\geq r\}$, quel que soit $a\in \Z_p$; c'est un $L$-Banach orthonormalisable et on note $\D_{r}(\Z_p,L)$ le $L$-dual de $\LA_{r}(\Z_p,L)$.
 
Si $n\in \N$ et $n\geq 1$, on note $r_n\geq 1$ le plus petit entier tel que $(p-1)p^{r_n}>n$. L'espace des poids $\sW$ admet le recouvrement admissible $\{\sW_n\}_{n\geq 0}$, où $\sW_n=\mu_{p-1}\times \Spm C_n$ avec $C_n$ le sous-anneau de $\Q_p[[T_1-1]]$ consistant des fonctions analytiques sur le disque $v_p(T_1-1)\geq \frac{1}{n}$.  
Pour chaque $\sW_n$, on note $\LA_{r,n}=\LA_{r}(\Z_p,\cO(\sW_n))$. Plus généralement, si $W$ est un ouvert affinoïde de $\sW$, on note $\LA_{r,W}$ le $\cO(W)$-algèbre de Banach $\LA_r(\Z_p, \cO(W))$ avec $r>0$. On note $\LA_{r,W}^0$ la boule unité de $\LA_{r,W}$.
\begin{lemma}\label{1-cocycle} 
Si $r\geq r_n$, la fonction $\rho_j^{\mathrm{univ}}$ est un $1$-cocycle sur $\I_0(p)$ à valeurs dans $\LA_{r,n}$. De plus, elle est une unité de $\LA_{r,n}$. Plus généralement, si $W$ est un ouvert affinoïde de $\sW$, il existe un nombre rationnel $r_W>0$, tel que pour tout $r\geq r_W$, la fonction $\rho_j^{\mathrm{univ}}$ est à valeurs le groupes des unités de l'anneau $\LA_{r,W}$.  
\end{lemma}
  


Ce lemme nous permet d'appliquer la construction dans \cite[\S 2.2]{WangI}, à la donnée $(\LA_{r,W},\rho^{\univ}_j )$; il en résulte une $\cO(W)$-représentation de Banach à droite $\D_{r,W,j}$ de $\I_0(p)$, qui est isomorphe à $\D_r(\Z_p,L)\hat{\otimes}\cO(W)$ comme $\cO(W)$-module, munie d'une application de spécialisation:
\[\mathrm{Sp}_{k,j}: \D_{r,W,j}\ra V_{k,j+2},\] 
donnée par la même formule dans le cas $\mathbf{Mes}^{\mathrm{bd}}_{\sW,j}$.


On note $\widehat{\mathbf{PD}}_r$ le sous $\Z_p$-module  de $\Q_p[[T]]$ des 
$\sum\limits_{n\in \N}a_n T^n/[\frac{n}{p^r}]!$, où les $a_n$ sont dans $\cO_L$. Le $\cO(W)$-module de Banach $\D_{r,W,j}$ est caractérisé complètement par la transformée d'Amice $\mu\mapsto \int_{\Z_p}(1+T)^x\mu$ et il est isomorphe à $\widehat{\mathbf{PD}}_r\hat{\otimes}_{\Z_p} \cO(W)$. 
On note $\D_{r,W,j}^+$ la boule unité de $\D_{r,W,j}$. Elle est stable sous l'action de $\I_0(p)$ par le lemme ci-dessus et on a un isomorphisme $\D_{r,W,j}^+\cong \widehat{\mathbf{PD}}_r\hat{\otimes}\cO(W)^+$, où $\cO(W)^+$ est la boule unité de $\cO(W)$. 
Ces représentations $\{\D_{r,W,j}\}_{r}$ forment un système projectif de représentations de Banach de $\I_0(p)$, qui sera utilisé dans la construction de la courbe de Hecke d'Ash-Stevens.
\subsection{Famille de systèmes d'Euler de Kato sur $\sW$ et ses variantes}
 Soit $K$ le sous-groupe compact de $\GL_2(\hat{\Z})$ défini comme suit:
on pose un sous-groupe de $\GL_2(\hat{\Z}^{]p[})$:
\[K^{]p[}= \{ (\begin{smallmatrix}a& b \\  c& d\end{smallmatrix})\in \GL_2(\hat{\Z}^{]p[}): c\equiv d-1\equiv 0\mod N \} ;   \]
et on note $K=K^{]p[}\times\I_0(p)$.
Alors son image sous l'application de détermiant est $\hat{\Z}^*$ et $K\cap \SL_2(\Z)= \Gamma(N;p)$. On note $\bar{\Gamma}(N;p)$ le complété profini de $\Gamma(N;p)$. 
La fonction caractéristique $\phi_K$ de $K$ est une fonction localement constante sur $\bM_2^{(p)}$ invariante sous l'action de $\tilde{K}$, où $\tilde{K}$ est l'image inverse de $K$ dans $\Pi_\Q$. 

Dans ce paragraphe, on améliore la construction de système d'Euler de Kato sur l'espace des poids dans $\cite{WangI}$, de sorte que, si $c,d\in \hat{\Z}^*$, on construct un élément 
\[z_{\Kato,c,d,K}(\nu_j)\in  \rH^2(\tilde{K}, \fD_0(\bM_2^{(p)},\mathbf{Mes}_{\sW,j}^{\mathrm{bd}}(2)) ).\]
\subsubsection*{L'opérateur $A_{c,d}$ revisité}
À partir des unités de Siegel, on construit (cf. \cite{PC1},\cite{Wang}) une distribution
alg\'ebrique $z_{\textrm{Siegel}}$ sur $\A_f^2-(0,0)$ \`a
valeurs dans $\Q\otimes(\cM(\bar{\Q})[\frac{1}{\Delta}])^{*}$, o\`u $\Delta=q\prod_{n\geq 1}(1-q^n)^{24}$ est la forme modulaire de poids $12$. La
distribution $z_{\textrm{Siegel}}$ est invariante sous l'action du groupe
$\Pi_{\Q}$.

La th\'eorie de Kummer $p$-adique nous fournit un \'el\'ement
\[z_{\textrm{Siegel}}^{(p)}\in \rH^1(\Pi_{\Q},\fD_{\alg}(\A_f^{2}-(0,0),\Q_p(1))).\]
Par cup-produit et
restriction \`a $\bM_2(\A_f)^{(p)}\subset(\A_f^{2}-(0,0))^2$,
 on obtient une distribution alg\'ebrique:
\[z_{\Kato}\in\rH^2(\Pi_{\Q},\fD_{\alg}(\bM_2(\A_f)^{(p)},\Q_p(2))).\]

L'espace topologique localement profini $\bM_2(\A_f)^{(p)}$ est muni de deux actions de $\GL_2(\hat{\Z})$ à gauche et à droite respectivement. Ceci induit deux actions de $\GL_2(\hat{\Z})$ sur $\fD_{\alg}(\bM_2(\A_f)^{(p)},\Q_p(2)) $ et l'action de $\Pi_\Q$ à droite est à traver son quotient $\GL_2(\hat{\Z})$ à droite. 
Alors, le $\Q_p$-espace $\rH^2(\Pi_{\Q},\fD_{\alg}(\bM_2(\A_f)^{(p)},\Q_p(2)))$ est un $\Z_p[\T(\hat{\Z})]$-module à gauche, où $\T$ est le tore du groupe algébrique $\mathbf{GL}_2$.
Soient $c,d\in \hat{\Z}^*$. On note $A_{c,d}$ l'élément dans $\Z_p[\T(\hat{\Z})]$ qui correspond à la distribution algébrique $(c_p^2\delta_1-\delta_{(\begin{smallmatrix}c^{-1}&0\\ 0& 1\end{smallmatrix})})(d_p^2\delta_1-\delta_{(\begin{smallmatrix}1&0\\ 0& d^{-1}\end{smallmatrix})})$ sur $\T(\hat{\Z})$, où $\delta_x$ désigne la masse de Dirac en $x$. On note 
\[z_{\Kato,c,d}= A_{c,d}z_{\Kato}\in \rH^2(\Pi_\Q, \fD_{\alg}(\bM_2(\A_f)^{(p)}, \Q_p(2) )).\]
\begin{lemma}On a $z_{\Kato,c,d}\in \rH^2(\Pi_\Q, \fD_{\alg}(\bM_2(\A_f)^{(p)}, \Z_p(2) ))$, ce qui nous permet de le voir comme un élément dans  $\rH^2(\Pi_\Q, \fD_{0}(\bM_2(\A_f)^{(p)}, \Z_p(2) ))$.
\end{lemma}
\begin{proof} Il est facile à vérifier que la construction de $z_{\Kato,c,d}$ ci-dessus coïncide avec celle dans $\cite[\S 2.3.2]{Wang}$ et la démonstration du lemme se trouve dans $\cite[\S 2.3.2]{Wang}$.
\end{proof}

 
 On note $\bM_2^{(p)}=\bM_2(\Q\otimes\hat{\Z}^{]p[})\times \I_0(p)$. Comme la fonction caractéristique $\phi_K$ de $K$ est une fonction localement constante sur $\bM_2^{(p)}$ invariante sous l'action de $\tilde{K}$,  une torsion à la Soulé (cf. \cite[\S 2.3]{WangI}) nous
fournit enfin des éléments de Kato
 \begin{equation*}
 \begin{split}
 &z_{\Kato,c,d,K}(k,j)=\phi_K\cdot(e_1^{k-2}t^{-j})*x_p\otimes z_{\Kato,c,d} \in \rH^2(\tilde{K}, \fD_0(\bM_2^{(p)},V_{k,j}) ) \\
 (\text{resp.} &z_{\Kato,c,d,K}(\nu_j)=\phi_K\cdot(\nu_j*x_p)\otimes z_{\Kato,c,d} \in \rH^2(\tilde{K}, \fD_0(\bM_2^{(p)},\mathbf{Mes}_{\sW,j}^{\mathrm{bd}}(2)) )).
\end{split}
 \end{equation*}


La proposition suivante est une conséquence directe du \cite[théorème 2.19]{WangI}:
\begin{prop} Si $1\leq j\in\N$, alors pour tout entier $k\geq 1+j$, on a 
\[\mathrm{Sp}_{k,j}(z_{\Kato,c,d,K}(\nu_j))=z_{\Kato,c,d,K}(k,j) .\]
\end{prop}
  \subsection{Projecteur d'Iwasawa}\label{stra_Col}

Soit $\hat{\Gamma}$ un sous-groupe de congruence de $K$ tel que l'application $\xymatrix{\hat{\Gamma}\ar[r]^{\det}&\hat{\Z}^*}$ est surjective, et on note $\tilde{\Gamma}$ l'image inverse de $\hat{\Gamma}$ dans $\Pi_\Q$ via l'application $\Pi_\Q\ra \GL_2(\hat{\Z})$ et $\bar{\Gamma}$ son intersection avec $\Pi_{\bar{\Q}}$, qui est le complété profini de $\Gamma=\hat{\Gamma}\cap\SL_2(\Z)$. On a une suite exacte de groupes:
\[1\ra \bar{\Gamma}\ra \tilde{\Gamma}\ra \cG_\Q\ra 1.\]
Soit $M\in \N$ tel que $\Gamma(M)\subset \Gamma$. 
On note $\Sigma=\{l\in \cP, l\mid Mp\}$ et $\bar{\Q}_\Sigma$ l'extension maximale de $\Q$ non ramifiée en dehors de $\Sigma$. On note $\cG_{\Q,\Sigma}$ le group de Galois de $\bar{\Q}_\Sigma$ sur $\Q$.

Soit $A$ une $\Z_p$-algèbre locale complete noethérienne de corps résiduel $\F_p$ et soit $\m$ son idéal maximal. Soit $V$ un $A\otimes_{\Z_p}\Q_p$-module muni d'une action de $\tilde{\Gamma}$ agissant à travers son quotient dans $\GL_2(\Z_p)$, tel que, il existe un $A$-sous-module  $V^+$ de rang fini stable sous l'action de $\tilde{\Gamma}$ tel que 
\[V^+=\plim_lV^+/\m^l \text{ et } V=V^+\otimes_{\Z_p} \Q_p.\] 

Soit $z\in \rH^2(\tilde{\Gamma}_0(p), \fD_0(\bM_2^{(p)},V))$. Si $\phi$ est une fonction localement constante sur 
$\bM_2(\hat{\Z})^{(p)}$ à valeurs dans $\Z$ invariante sous l'action de $\tilde{\Gamma}$, on définit une mesure $z_{\phi,0}\in \rH^2(\tilde{\Gamma}_0(p), \fD_0(\Z_p^*,V) )$ par 
\[ \int_{\Z_p^*}\psi z_{\phi,0}=\int_{\bM_2(\hat{\Z})^{(p)}}\psi(\det x_p) \phi(x) z, \text{ si } \psi \in\cC^0(\Z_p^*,\Z_p) . \]
On note $G_n=\Z_p^*/(1+p^n\Z_p)$.
La suite spectrale de Grothendieck pour la cohomologie continue \cite[(3.4) corollary]{UJ} nous donne une application:
\[ \rH^2(\tilde{\Gamma}, \fD_0(\Z_p^*,V))\ra  (\rH^1(\cG_\Q, \plim_{n,l}\rH^1(\bar{\Gamma},  \Z_p[G_n]\otimes V^+/\m^l))\otimes\Q_p, \]
 qui se factorise par $(\rH^1(\cG_{\Q,\Sigma}, \plim_{n,l}\rH^1(\bar{\Gamma},  \Z_p[G_n]\otimes V^+/\m^l))\otimes\Q_p.$
\begin{lemma} On a un isomorphisme:
\[(\rH^1(\cG_{\Q,\Sigma}, \plim_{n,l}\rH^1(\bar{\Gamma},  \Z_p[G_n]\otimes V^+/\m^l))\otimes\Q_p\cong\rH^1_{\Iw}(\cG_{\Q,\Sigma}, \rH^1(\bar{\Gamma}, V)),\]
où $\rH^1_{\Iw}(\cG_{\Q,\Sigma}, \rH^1(\bar{\Gamma}, V))=\rH^1(\cG_{\Q,\Sigma}, \fD_0(\Z_p^*,\rH^1(\bar{\Gamma}, V)))$.
\end{lemma}
\begin{proof}Comme $\bar{\Gamma}$ agit trivalement sur $\Z_p^*$, on a un isomorphisme:
\[ \rH^1(\bar{\Gamma},   \Z_p[G_n]\otimes V^+/\m^l ) )\cong  \Z_p[G_n]\otimes\rH^1(\bar{\Gamma},V^+/\m^l) .\]
Comme la cohomologie de $\bar{\Gamma}$ dans un module fini est un groupe fini, le système projectif $\{\rH^1(\bar{\Gamma},V^+/\m^l)\}_{l}$ est un système de Mittag-Leffler. Ceci implique que 
\[\plim_{l}\rH^1(\bar{\Gamma},  V^+/\m^l ) \cong \rH^1(\bar{\Gamma}, V^+).\] Ceci nous permet de conclure la preuve.

\end{proof}

En composant les applications obtenus ci-dessus,
on déduit d'un morphisme, appelé le projecteur d'Iwasawa,
  \[ \pi_{\phi,V}: \rH^2(\tilde{\Gamma}, \fD_0(\bM_2^{(p)},V) )\ra \rH^1_{\Iw}(\cG_{\Q,\Sigma}, \rH^1(\bar{\Gamma}, V)).\]
La construction ci-dessus applique en particulier aux représentations $\mathbf{Mes}^{\mathrm{bd}}_{\sW,j}(2)$ et $V_{k,j}$ et aux fonctions localement constantes sur $\bM_2^{(p)}$ invariante sous l'action de $\tilde{K}$. 

 On a les groupes $\tilde{K}$, $\bar{K}$ associés à $K$ comme ci-dessus. On fixe $\Sigma=\{l\in \cP, l\mid Np\}$ dans la suite.
 Soit $\phi$ une fonction localement constante sur 
$\bM_2(\hat{\Z})^{(p)}$ à valeurs dans $\Z$ invariante sous l'action de $\tilde{K}$.
 On note les projecteurs d'Iwasawa associés à $\mathbf{Mes}^{\mathrm{bd}}_{\sW,j}(2)$ et sur $V_{k,j}$ par $\pi_{\phi,\sW,j}$ et $\pi_{\phi,k,j}$ respectivement.  

L'application de spécialisation $\mathrm{Sp}_{k,j}:\mathbf{Mes}^{\mathrm{bd}}_{\sW,j}(2)\ra V_{k,j}$ induit une application de spécialisation 
\[\rH^1_{\Iw}(\cG_{\Q,\Sigma},  \rH^1(\bar{\Gamma}(N;p), \mathbf{Mes}^{\mathrm{bd}}_{\sW,j}(2)))\ra  \rH^1_{\Iw}(\cG_{\Q,\Sigma},  \rH^1(\bar{\Gamma}(N;p), V_{k,j}))\] notée encore par $\mathrm{Sp}_{k,j}$. Le théorème suivant est un résumé de la construction ci-dessus:
\begin{theo}\label{projection}
(1) On a le diagramme commutatif suivant:
\[\xymatrix{\rH^2(\tilde{K}, \fD_0(\bM_2^{(p)},\mathbf{Mes}^{\mathrm{bd}}_{\sW,j}(2) ) )\ar[d]^{\mathrm{Sp}_{k,j}}\ar[rr]^{\pi_{\phi,\sW,j}}& & \rH^1_{\Iw}(\cG_{\Q,\Sigma},  \rH^1(\bar{\Gamma}(N;p), \mathbf{Mes}^{\mathrm{bd}}_{\sW,j}(2)))\ar[d]^{\mathrm{Sp}_{k,j}}\\  \rH^2(\tilde{K}, \fD_0(\bM_2^{(p)},V_{k,j}))\ar[rr]^{\pi_{\phi,k,j}}& & \rH^1_{\Iw}(\cG_{\Q,\Sigma},  \rH^1(\bar{\Gamma}(N;p), V_{k,j}))}.\]
(2) On note $z_{\Kato,c,d,\phi}(\sW,j)$ (resp. $z_{\Kato,c,d,\phi}(k,j)$) l'image de $z_{\Kato,c,d,K}(\nu_j)$ (resp. $z_{\Kato, c,d,K}(k,j)$) sous le projecteur d'Iwasawa $\pi_{\phi,\sW,j}$ (resp. $\pi_{\phi,k,j}$). Alors, on a 
\[\mathrm{Sp}_{k,j}(z_{\Kato,c,d,\phi}(\sW,j))= z_{\Kato,c,d,\phi}(k,j).\]
\end{theo}

\section{Construction de la famille de systèmes d'Euler de Kato sur la courbe de Hecke cuspidale}\label{const5}
On a un élément de Kato $z_{\Kato, c,d, \phi}(\sW,j)\in \rH^1_{\Iw}(\cG_{\Q,\Sigma},  \rH^1(\bar{\Gamma}(N;p), \mathbf{Mes}^{\mathrm{bd}}_{\sW,j}(2)))$ et on veut en déduire un élément de Kato sur $\fC^0$. Le problème est de fabriquer une famille de représentations galoisiennes au-dessus de $\fC^0$ à partir de $\rH^1(\bar{\Gamma}(N;p), \mathbf{Mes}^{\mathrm{bd}}_{\sW,j}(2))$. Pour effectuer ça, on suit de près la méthode d'Ash-Stevens \cite{AS}, et Bellaïche (\cite{Be}, \cite{Be1}). La stratégie est la suivante:
\begin{itemize}
\item[$\bullet$] on étend les coefficients $ \mathbf{Mes}^{\mathrm{bd}}_{\sW,j}$ en $ \D_{r,W,j}$ et on descend du groupe profini $\bar{\Gamma}(N;p)$ au groupe discret $\Gamma(N;p)$;
\item[$\bullet$] on utilise la décomposition par les pentes de l'opérateur $U_p$ pour construire un faisceau sans torsion au-dessus le morceau local de $\fC^0$ muni d'une action continue de $\cG_\Q$ et enfin on globalise par un processus standard \cite[\S 5]{KB}, vérifié par Bellaïche \cite[chapter II]{Be1} pour notre cas. 
\end{itemize}



\subsection{Cohomologie du groupe profini et cohomologie du groupe discret}

Rappelons que, pour tout ouvert affinoïde $W=\Spm R \subset\sW$, il existe un nombre $r_W$, tel que, pour tout $r\geq r_W$, on a une $\cO(W)$-représentation de Banach $\D_{r,W,j}$ (cf. \S \ref{system}) de $\I_0(p)$. 
Pour tout $r\geq r_W$, on a un morphisme naturel de $\cO^{\mathrm{bd}}(\sW)$-modules $\I_0(p)$-équivariant
\[\mathbf{Mes}^\mathrm{bd}_{\sW,j}\ra \mathbf{Mes}^{\mathrm{bd}}_{\sW,j}\otimes_{\cO^{\mathrm{bd}}(\sW)}R\ra \D_{r,W,j}.\]
Ceci induit un morphisme 
\[  \rH^1_{\Iw}(\cG_{\Q,S},  \rH^1(\bar{\Gamma}(N;p), \mathbf{Mes}^{\mathrm{bd}}_{\sW,j}(2)))\ra  \rH^1_{\Iw}(\cG_{\Q,S},  \rH^1(\bar{\Gamma}(N;p), \D_{r,W,j}(2))).\]
Si $k\in W\cap\N $, alors l'application de spécialisation $\mathrm{Sp}_{k,j}:\D_{r,W,j}(2)\ra V_{k,j}$ induit un morphisme de spécialisation
\[\mathrm{Sp}_{k,j}:\rH^1_{\Iw}(\cG_{\Q,S},  \rH^1(\bar{\Gamma}(N;p), \D_{r,W,j}(2)))\ra  \rH^1_{\Iw}(\cG_{\Q,S},  \rH^1(\bar{\Gamma}(N;p), V_{k,j})), \]
tel que, le diagramme suivant est commutatif:
\[\xymatrix{ \rH^1_{\Iw}(\cG_{\Q,S},  \rH^1(\bar{\Gamma}(N;p), \mathbf{Mes}^{\mathrm{bd}}_{\sW,j}(2)))\ar[r]\ar[dr]^{\mathrm{Sp}_{k,j}}&\rH^1_{\Iw}(\cG_{\Q,S},  \rH^1(\bar{\Gamma}(N;p), \D_{r,W,j}(2)))\ar[d]^{\mathrm{Sp}_{k,j}}\\  & \rH^1_{\Iw}(\cG_{\Q,S},  \rH^1(\bar{\Gamma}(N;p), V_{k,j}))}.\]

La proposition suivante montre que le $\cO(W)$-module $\rH^1(\Gamma(N;p), \D_{r,W,j})$ (resp. le $\Q_p$-espace $\rH^1(\Gamma(N;p),\D_{r,k,j}))$ est une $W$-représentation (resp. $\Q_p$-représentation) du groupe de Galois $\cG_{\Q,S}$.

\begin{prop}\label{betti-etale}On a les isomorphismes de modules de Hecke suivant: \\
$(1)$ $\rH^1(\Gamma(N;p), \D_{r,W,j})\cong (\plim_{n}\rH^1(\bar{\Gamma}(N;p), \D_{r,W,j}^+/p^n))\otimes\Q_p\cong\rH^1(\bar{\Gamma}(N;p), \D_{r,W,j})$, \\
$(2)$ $\rH^1(\Gamma(N;p), \D_{r,k,j})\cong \rH^1(\bar{\Gamma}(N;p), \D_{r,k,j}).$
\end{prop}
\begin{proof} Le $(1)$ et $(2)$ se démontrent de la même manière. On ne démontre que le premier.
Rappelons que on a $\D_{r,W,j}\cong (\plim_n \D_{r,W,j}^+/p^n)\otimes\Q_p$. Pour simplifier la notation, on note $\Gamma=\Gamma(N;p)$ et $\bar{\Gamma}=\bar{\Gamma}(N;p)$. On a une suite exacte
\[0\ra (\plim^1_n\rH^0(\bar{\Gamma}, \D_{r,W,j}^+/p^n))\otimes\Q_p\ra \rH^1(\bar{\Gamma}, \D_{r,W,j})\ra  (\plim_{n}\rH^1(\bar{\Gamma}, \D_{r,W,j}^+/p^n))\otimes\Q_p\ra 0.\]
Pour démontrer l'isomorphisme $(\plim_{n}\rH^1(\bar{\Gamma}, \D_{r,W,j}^+/p^n))\otimes\Q_p\cong\rH^1(\bar{\Gamma}, \D_{r,W,j})$, on se ramène à montrer que $(\plim^1_n\rH^0(\bar{\Gamma}, \D_{r,W,j}^+/p^n))\otimes\Q_p=0$. Pour tout $m\geq 1$, on a 
\[\rH^0(\bar{\Gamma}, \D_{r,W,j}^+/p^n)=\rH^0(\Gamma(N;p),\D_{r,W,j}^+/p^n )\subset (\D_{r,W,j}^+/p^n)^{U_m}\] avec $U_m=(\begin{smallmatrix}1& p^m\Z_p\\ 0&1\end{smallmatrix})$ le sous-groupe unipotent de $\GL_2(\Z_p)$. 

Si $-v_p(\log(1+T))< m\leq n-v_p(\log(1+T))$, on note $\partial_m=\lim\limits_{l\ra \infty}\frac{u_m^{p^l}-1}{p^l}$ l'opérateur différentiel associé à un générateur $u_m$ du groupe $p$-adique analytique $U_m$. Un calcul immédiat montre que, pour tout $\mu\in \D_{r,W,j}$, on a 
\[\partial_m \cA_\mu= p^m\log (1+T) \cA_\mu, \text{ où } \cA_\mu \text{ est la transformé d'Amice de } \mu.\] 
Ceci implique que, si $-v_p(\log(1+T))< m< n-v_p(\log(1+T))$,  les éléments dans $ (\D_{r,W,j}^+/p^n)^{\partial_m=0}$ sont de $p^{m}$-torsion, pour tout les $n$.
On en déduit que  $(\plim^1_n\rH^0(\bar{\Gamma}, \D_{r,W,j}^+/p^n))\otimes\Q_p=0$. 

D'autre part, d'après \cite[p.15 Exercises]{Se}, on a un isomorphisme 
\[\rH^i(\Gamma, \D_{r,W,j}^+/p^n)\cong \rH^i(\bar{\Gamma}, \D_{r,W,j}^+/p^n).\] Ceci nous permet de conclure la preuve. 
 \end{proof}

\subsection{Projection sur $\fC^0$}
Soit $v$ un nombre réel. Il existe un ouvert $W\subset\sW$ et un $r_W>0$ tels que l'on peut appliquer la technique\footnote{Cette technique est introduite par Ash-Stevens \cite[proposition 4.1.2]{AS}. L'avantage est que si un complexe admet une telle décomposition, alors sa cohomologie aussi l'admet. } de la décomposition de pente $\leq v$ aux paires $(\rH^1(\Gamma(N;p), \D_{r_W,W,j}), U_p)$. 
En plus, il existe un ouvert $W$ adapté à $v$ (i.e. on a un isomorphisme de $\cO(W)$-modules $\rH^1(\Gamma(N;p), \D_{r,W,j})^{\leq v}\cong \rH^1(\Gamma(N;p), \D_{r,W,j})^{\leq v}$ pour tout $r,r'\geq r_W$). Fixons $W=\Spm R$ un ouvert affinoïde de $\sW$ adapté à $v$. 

Soit $\cH_N$ la $\Z$-algèbre commutative engendrée par les opérateurs de Hecke abstraits $T(l)$ pour tout premier $l$ avec $(l,Np)=1$, l'opérateur d'Atkin-Lehner $U_p$ et les opérateurs de diamant $\langle a\rangle$ pour $a\in (\Z/N\Z)^*$. 
Si $k\in\N$, on note $S_{k+2}(\Gamma(N;p))$ (resp.  $S_{k+2}^{\dag}(\Gamma(N;p)))$ l'espace des formes modulaires classiques cuspidales (resp. surconvergentes cuspidales) de niveau $\Gamma(N;p)$ et de poids $k+2$. Ils sont munis d'une action de $\cH_N$. Pour $M$ un des espaces des formes modulaires ci-dessus, on note $M^{< v}$ le sous-espace de pente\footnote{La pente d'une forme modulaire est la valuation $p$-adique de la valeur propre de $U_p$.} $<r$. On a le résultat de classicité (cf. \cite[corollary 2.6]{Be}) suivant :

\begin{prop}\label{classicité}On a un isomorphisme de Hecke modules:
\[S^\dag_{k+2}(\Gamma(N;p))^{< k+1}= S_{k+2}(\Gamma(N;p))^{<k+1}.\]
\end{prop}

On note $\fC^0_{W,v}$ le morceau local de $\fC^0$, muni d'un morphisme fini plat de poids $\kappa:\fC^0_{W,v}\ra W$, paramétrant les formes modulaires surconvergentes cuspidales de pente $\leq v$. 

\begin{theo}Il existe un faisceau sans torsion $\sV_{W,v}$ sur $\fC^0_{W,v}$ muni d'une action $\fC^0_{W,v}$-linéaire continue du groupe de Galois $\cG_\Q$, tel que, si $f_\alpha\in \fC^0_{W,v}$, on a  $\mathrm{Ev}_{f_\alpha}(\sV_{W,v})=V_{f_\alpha}$. En plus, on a une projection explicite Galois équivariante de $\rH^1(\G(N;p),\D_{r,W,j})$ sur $\sV_{W,v}(1-j)$.
\end{theo}
\begin{proof}

Comme l'action de groupe de Galois commute à celle de $U_p$, la décomposition de pente $\leq v$ pour l'opérateur $U_p$ nous donne un diagramme commutatif:
\[\xymatrix{ \rH^1_{\Iw}(\cG_{\Q,\Sigma},  \rH^1(\G(N;p), \D_{r,W,j}(2)))\ar[r]\ar[d]^{\mathrm{Sp}_{k,j}}& \rH^1_{\Iw}(\cG_{\Q,\Sigma},  \rH^1(\G(N;p), \D_{r,W,j}(2) )^{\leq v} )\ar[d]^{\mathrm{Sp}_{k,j}}\\  \rH^1_{\Iw}(\cG_{\Q,\Sigma},  \rH^1(\G(N;p), V_{k,j}))\ar[r]& \rH^1_{\Iw}(\cG_{\Q,\Sigma},  \rH^1(\G(N;p), V_{k,j})^{\leq v})}.\]

La matrice $(\begin{smallmatrix}1& 0\\ 0& -1\end{smallmatrix})$ définit une involution $\iota$ sur $\rH^{1}(\G(N;p),V)$, où $V$ est une des représentations: $\D_{r,W,j}$,  $\D_r(k,j)$, ou $V_{k,j}$. On note $M^\pm=\rH^{1,\pm}(\G(N;p),\D_{r,W,1})$ le sous-module de $\rH^{1}(\G(N;p),\D_{r,W,1})$ des éléments fixés ou multipliés par $-1$ sous cette involution, qui sont des $\cH_N$-modules car l'involution commute à l'action de $\cH_N$. 
On en déduit deux morphismes de $\Q_p$-algèbres $\cH_N\ra \End_{\cO(W)}(M^{\pm})$ et on note $\T^{\pm}_{W,v}$ la sous-$\cO(W)$-algèbre de $\End_{\cO(W)}(M^\pm)$ engendrée par l'image de $\cH_N$ respectivement, qui est de rang fini comme $M^{\pm}$ est un $\cO(W)$-module de rang fini. En plus, $\T^{\pm}_{W,v}$ sont sans-torsion. Ceci nous fournit deux courbes rigides $\fC^{\pm}_{W,v}:=\Spm \T^{\pm}_{W,v}$ munies de deux morphismes de poids finis plats $\kappa^{\pm}: \fC^{\pm}_{W,v}\ra W $; elles sont les morceaux locaux de la courbe d'Ash-Stevens $\fC^{\pm}$ respectivement.

La proposition suivant donne une comparaison entre $\fC^{\pm}_{W,v}$ et $\fC_{W,v}^0$, qui est un analogue du \cite[théorème 3.27]{Be} et se démontre exactement de la même manière.

\begin{prop}\label{comparaison} Soient $v,W$ comme ci-dessus. \\
$(1)$ Il existe une immersion fermée $\fC^0_{W,v}\subset\fC^{\pm}_{W,v}$, qui est compatible avec les morphismes de poids $\kappa^{\pm}$ et  $\kappa$. De plus, toutes ces courbes sont réduites.\\
$(2)$ Soit $k\in \Z$. Il existe une injection de modules de Hecke
\[S_{k}^\dag(\Gamma(N;p))^{\mathsf{ss},\leq v}\subset  \mathbf{Ev}_k(M^{\pm})^\mathbf{ss}\subset M_k^\dag(\Gamma(N;p))^{\mathsf{ss},\leq v},\]
où $\mathsf{ss}$ signifie la semi-simplification comme module de Hecke.\\
$(3)$ Soit $f_\alpha\in \fC^0_{W,v} $ une forme classique cuspidale raffinée. Soit $V$ un Hecke module. On note $V_{(f_\alpha)}$ le sous-espace propre généralisé de $V$ pour le système des valeurs de Hecke associées à $f_\alpha$. On a 
\[\dim(S^\dag_k(\Gamma(N;p))_{(f_\alpha)})=\dim ( M^{\pm}_{(f_\alpha)} ).\]

\end{prop}
On note $\iota^{\pm}: \fC^{\pm}_{W,v}\ra \fC^{+}_{W,v}\cup\fC^{-}_{W,v}$, et on définit le faisceau $\sV_{W,v}$ sans torsion sur $\fC^0_{W,v}$ en prenant la restriction du faisceau $\iota^{+}_{*}M^+\oplus \iota^{-}_{*}M^{-}$ sur $\fC^{+}_{W,v}\cup\fC^{-}_{W,v}$ à $\fC^0_{W,v}$. D'autre part, le faisceau $\iota^{+}_{*}M^+\oplus \iota^{-}_{*}M^{-}$ est la faisceautisé du module $\rH^1(\G(N;p), \D_{r,W,1})^{\leq v}$, et donc un faisceau de représentations galoisiennes de $\cG_{\Q}$ d'après la proposition \ref{betti-etale}.

 Soit $f_\alpha\in \fC^0_{W,v}$ une forme cuspidale raffinée non-critique de poids $k$ de pente $v_p(\alpha)$. D'après le (3) du proposition $\ref{comparaison}$, on a un isomorphisme de représentations galoisiennes 
\[\mathrm{Ev}_{f_\alpha}(\sV_{W,v})\cong \rH^1(\Gamma(N;p), V_{k,1})^{\leq v}_{\pi_{f_\alpha}}=V_{f_\alpha}.\] \end{proof}

L'application d'évaluation $\mathrm{Ev}_{f_\alpha}:\sV_{W,v}\ra V_{f_\alpha}$ induit un morphisme de spécialisation de $\rH^1_{\Iw}(\cG_{\Q,\Sigma},\sV_{W,v})$ dans $\rH^1_{\Iw}(\cG_{\Q,\Sigma}, V_{f_\alpha})$, noté $\mathrm{Sp}_{f_\alpha}$. La proposition suivante est une conséquence immédiate de la construction de la projection sur $\fC^0_{W,v}$ et de la construction de $\mathrm{Sp}_{f_\alpha}$.
\begin{prop}\label{pj}Soit $f_\alpha\in \fC^0_{W,v}$ une forme raffinée classique de poids $k$ de pente $v_p(\alpha)$. Le diagramme suivant est commutatif:
\[\xymatrix{ \rH^1_{\Iw}(\cG_{\Q,\Sigma},  \rH^1(\Gamma(N;p), \D_{r,W,j}(2) )^{\leq v} ) \ar[r] \ar[d]^{\mathrm{Sp}_{k,j}}& \rH^1_{\Iw}(\cG_{\Q,\Sigma},\sV_{W,v}(1-j))\ar[d]^{\mathrm{Sp}_{f_{\alpha}}} \\ \rH^1_{\Iw}(\cG_{\Q,\Sigma},  \rH^1(\Gamma(N;p), V_{k,j})^{\leq v})\ar[r]^{\pi_{f_\alpha}}&\rH^1_{\Iw}(\cG_{\Q,\Sigma}, V_{f_\alpha}(1-j))
  }.\]
\end{prop}
On note $z_{\mathrm{Kato},c,d,\phi}(\fC_{W,v}^0)$ (resp. $z_{\mathrm{Kato},c,d,\phi}(f_\alpha)$) l'image de $z_{\mathrm{Kato},c,d,K}(\nu_1)$ (resp. $z_{\mathrm{Kato},c,d,K}(k,1)$) dans $ \rH^1_{\Iw}(\cG_{\Q,S},\sV_{W,v})$ (resp. $\rH^1_{\Iw}(\cG_{\Q,S}, V_{f_\alpha})$) sous l'application de projection construite ci-dessus.

 Comme les morceaux locaux $\fC^0_{W,v}$ sont construites par la méthode d'Ash-Stevens, et de Bellaïche, on peut utiliser le processus standard de Buzzard \cite[\S 5]{KB}, révisité\footnote{Bellaïche montre que les morceaux locaux construites par la méthode d'Ash-Stevens, et de Bellaïche, forment une recourvement admissible.} par Bellaïche \cite[Chapter II]{Be1}, pour reconstruire la courbe de Hecke cuspidale $\fC^0$, ainsi que un faisceau $\sV$ de représentations galoisiennes sur $\fC^0$ en collant les morceaux locaux.   
\begin{theo}Il existe une section globale $z_{\Kato,c,d,\phi}(\fC^0)$ de $\rH^1_{\Iw}(\cG_{\Q,S}, \sV)$, telle que, si $f_\alpha\in \fC^0$ une forme raffinée classique de poids $k$ de pente $v_p(\alpha)$,  on a 
\[\mathrm{Sp}_{f_{\alpha}}(z_{\mathrm{Kato},c,d,\phi}(\fC^0))=z_{\mathrm{Kato},c,d,\phi}(f_\alpha).\]
\end{theo} 
 \begin{proof}Les éléments de Kato $z_{\Kato,c,d,\phi}(\fC_{W,v}^0)$ se colle en une section globale $z_{\Kato,c,d,\phi}(\fC^0)$ de $\rH^1_{\Iw}(\cG_{\Q,S}, \sV)$ puisqu'ils proviennent d'une section globale $z_{\Kato,c,d,K}(\nu_1)$ et les constructions d'éléments locaux sont unifiées. La propriété d'interpolation se déduit de la relation 
 \[\mathrm{Sp}_{k,j}(z_{\mathrm{Kato},c,d,\phi}(\nu_j))=z_{\mathrm{Kato},c,d,\phi}(k,j),\] des définitions de $z_{\mathrm{Kato},c,d,\phi}(\fC^0_{W,v})$ et $z_{\mathrm{Kato},c,d,\phi}(f_\alpha)$, et de la proposition \ref{pj}.
 \end{proof}

\section{Fonction L $p$-adique en deux variables}

\subsection{Les séries d'Eisenstein}\label{variant}
On note $\Dir(\ol{\Q})$ le $\ol{\Q}$-espace vectoriel des s\'eries de Dirichlet formelles \`a coefficients dans $\ol{\Q}$. Soit $A$ un sous-anneau de $\ol{\Q}$. On note $\Dir(A)$ le sous $A$-module de $\Dir(\ol{\Q})$ des s\'eries de Dirichlet formelles dont ses coefficients sont dans $A$. 

Soit $\alpha\in\Q/\Z$. On d\'efinit les s\'eries de Dirichlet formelles $\z(\alpha, s)$ et
$\z^{*}(\alpha,s)$, appartenant \`a $\Dir(\Q^{\cycl})$, par les formules:
\[\z(\alpha,s)=\sum\limits_{\substack{n\in\Q_{+}^{*}\\ n\equiv\alpha \mod\Z}}n^{-s} \text{ et } \z^{*}(\alpha,s)=\sum_{n=1}^{\infty}e^{2i\pi\alpha n}n^{-s}.\]

On définit les séries d'Eisenstein $E_{\alpha,\beta}^{(k)}$ et $F_{\alpha,\beta}^{(k)}$, éléments de $\cM^{\con}(\Q^{\cycl})$, par ses $q$-d\'eveloppements:

$(1)$ si $k\geq 1, k\neq 2$ et $\alpha, \beta\in\Q/\Z$, alors le $q$-d\'eveloppement $\sum\limits_{n\in\Q_{+}}a_nq^n$ de $E_{\alpha,\beta}^{(k)}$ est donn\'e par
\begin{equation*}
\sum_{n\in\Q_{+}^{*}}\frac{a_n}{n^s}=\z(\alpha,s)\z^{*}(\beta,s-k+1)+(-1)^{k}\z(-\alpha,s)\z^{*}(-\beta,s-k+1).
\end{equation*}
De plus, on a:
si $k\neq 1 $ et si $\alpha\neq 0$ $($resp. $\alpha=0)$, alors $a_0=0$ $($resp. $a_0=\z^{*}(\beta,1-k))$;\\
si $k=1$ et si $\alpha\neq
0$ $($resp. $\alpha=0)$, alors $a_0=\z(\alpha,0) ($resp.
$a_0=\frac{1}{2}(\z^{*}(\beta,0)-\z^{*}(-\beta,0))).$ \\
$(2)$ si $k\geq 1$ et $\alpha, \beta\in\Q/\Z~($si $k=2$, $(\alpha,\beta)\neq (0,0))$, alors le $q$-d\'eveloppement
$\sum\limits_{n\in\Q_{+}}a_nq^n$ de $F_{\alpha,\beta}^{(k)}$ est donn\'e par
\begin{equation*}
\sum_{n\in\Q_{+}^{*}}\frac{a_n}{n^s}=\z(\alpha,s-k+1)\z^{*}(\beta,s)+(-1)^{k}\z(-\alpha,s-k+1)\z^{*}(-\beta,s).
\end{equation*}
De plus, on a:
si $k\neq 1 $, alors $a_0=\z(\alpha,1-k)$, la valeur spéciale de la fonction zêta de Hurwitz;
si $k=1$ et si $\alpha\neq
0~($resp. $\alpha=0)$, alors $a_0=\z(\alpha,0)~($resp.
$a_0=\frac{1}{2}(\z^{*}(\beta,0)-\z^{*}(-\beta,0)))$ .

L'action de $\GL_2(\hat{\Z})$ induite par celle de $\Pi_\Q$ sur ces fonctions est donnée par les formules suivantes (cf. \cite[proposition 2.12]{Wang}):
 Si $\gamma=\bigl(\begin{smallmatrix}a & b\\ c &
d\end{smallmatrix}\bigr)\in\GL_2(\hat{\Z}), k\geq 1$ et
$(\alpha,\beta)\in(\Q/\Z)^2$, on a:
\begin{equation}\label{eis}E_{\alpha,\beta}^{(k)}*\gamma=E_{a\alpha+c\beta,b\alpha+d\beta}^{(k)}\text{ et } F_{\alpha,\beta}^{(k)}*\gamma=F_{a\alpha+c\beta,b\alpha+d\beta}^{(k)}.\end{equation}

Les relations de distribution des séries d'Eisenstein se traduisent en l'énoncé (cf. \cite[théorème 2.13]{Wang})\footnote{En profitons pour corriger une erreur dans l'énonce de \cite[théorème 2.13]{Wang}: en effet, pour donner un sens pour $z_{\mathrm{Eis}}(k)$, on utilise les deux isomorphismes $\A_f/\hat{\Z}\cong \Q/\Z$ et $\A_f^*\cong \Q^*\hat{\Z}^*$. Le dernier isomorphisme implique que $\A_f^*/\hat{\Z}^*\cong \Q_{+}^*$.   } suivant:  
\begin{prop}\label{eiskj}
Si $k\geq 1$, il existe une distribution alg\'ebrique $z_{\mathrm{Eis}}(k)$ $($resp. $z^{'}_{\mathrm{Eis}}(k))$ $\in \fD_{\alg}((\A_f)^2, \cM_k^{\con}(\Q_p^{\cycl}))$ v\'erifiant: quel que soient $r\in\Q^*_{+}$ et $(a,b)\in\Q^2$, on a
\begin{equation*}
\begin{split}
\int_{(a+r\hat{\Z})\times(b+r\hat{\Z})}z_{\mathrm{Eis}}(k)=r^{-k}E^{(k)}_{r^{-1}a,r^{-1}b}\
(\text{resp.} \int_{(a+r\hat{\Z})\times(b+r\hat{\Z})}z^{'}_{\mathrm{Eis}}(k)=r^{k-2}F^{(k)}_{r^{-1}a,r^{-1}b}.)
\end{split}
\end{equation*}
De plus, si $\gamma\in\GL_2(\A_f)$, alors
on a \[z_{\mathrm{Eis}}(k)*\gamma=z_{\mathrm{Eis}}(k) \text{ et } z_{\mathrm{Eis}}^{'}(k)*\gamma=|\det\gamma|^{1-k}z_{\mathrm{Eis}}^{'}(k).\]
\end{prop}

On peut identifier
$\A_f^2\times\A_f^2$ avec
$\bM_2(\A_f)$ via 
$((a,b),(c,d))\mapsto\bigl(\begin{smallmatrix}a & b\\ c &
d\end{smallmatrix}\bigr)$. 
En utilisant le fait que le produit de deux formes modulaires de poids $i$ et $j$ est une forme modulaire de poids $i+j$, si $k\geq 2$ et $1\leq j\leq k-1$, on définit une distribution $z_{\mathrm{Eis}}(k,j)$ appartient \`a $\fD_{\alg}(\bM_2(\A_f),\cM_k(\Q^{\cycl}_p))$ par la formule:
   \[z_{\mathrm{Eis}}(k,j)=\frac{1}{(j-1)!}z'_{\mathrm{Eis}}(k-j)\otimes z_{\mathrm{Eis}}(j).\] 
Si $c,d\in \hat{\Z}^{(p)}$, on note $A_{c,d}=(c_p^2-c_p^{2-(k-j)}\delta_{(\begin{smallmatrix}c^{-1}& 0\\ 0&1\end{smallmatrix})})(d_p^2-d_p^{j}\delta_{(\begin{smallmatrix}1&0\\ 0& d^{-1}\end{smallmatrix})})\in \Z_p[\T(\hat{\Z})^{(p)}]$ et on définit
\[z_{\mathrm{Eis},c,d}(k,j)= A_{c,d}z_{\mathrm{Eis}}(k,j).\]

\subsection{La méthode de Rankin-Selberg}
Dans la suite, on suppose que $N$ est un entier suffisantment grand tel qu'il existe un caractère auxiliaire $\chi$ de conducteur $N$ tel que $\chi(-1)=-1$ et $\chi^2$ est un caractère non-trivial de conducteur $N$. Fixons un tel caractère auxiliaire $\chi$ jusqu'à la fin de cet article. 

Soit $\eps$ un caractère de Dirichlet modulo $N$. Soit $f=\sum_{n\geq 1}a_nq^n\in S_k(\Gamma_0(N), \eps)$ une forme primitive.  On note $f^{*}$ la conjugu\'ee complexe de $f$ (i.e.
$f^{*}(z)=\ol{f(\bar{z})}=\sum\ol{a_n}q^n$ ). Le corps $\Q(f)=\Q(a_2,\cdots,a_n,\cdots)$ est une extension finie de $\Q$. De plus, on a $\ol{a_n}=\chi^{-1}(n) a_n$ quel que soit $n\in\N$ premier à $N$, et 
$ f*T(l)=a_lf, f*T'(l)=\ol{a_l}f, f*(\begin{smallmatrix}u^{-1}& 0\\ 0 &u\end{smallmatrix})=\eps(u)f,$
si $l\nmid N$ est un nombre premier et $u\in \hat{\Z}^*$. Soient $\alpha,\beta$ les racines du polyn\^ome $X^2-a_pX+\eps(p)p^{k-1}$. Si $v_p(\alpha)<k-1$, on pose $f_{\alpha}(\tau)=f(\tau)-\beta f(p\tau)$ le raffinement de $f$. 

On reprend la présentation de Colmez $\cite[\S 3.2]{PC1}$ sur la méthode de Rankin-Selberg dans $\S \ref{gene}$, on effectue les calculs pour la convolution de Rankin pour nos séries d'Eisenstein dans $\S \ref{cal}$, ce qui permet d'expliciter la période utilisée pour rendre algébriques les values spéciales de la fonction  L de $f_\alpha$. 


\subsubsection{Généralité}\label{gene}
Soit $M\geq 1$ un entier. Soit $\Gamma(M)=\{(\begin{smallmatrix}a&b\\ c&
d\end{smallmatrix})\in\SL_2(\Z), b\equiv c\equiv 0[M]\}$ et soit
$\chi_1$ et $\chi_2$ des caract\`eres de Dirichlet modulo $M$ (pas
n\'ecessairement primitifs). Soient $k\geq 2, 1\leq j\leq k-1$, et
\[f=\sum\limits_{n\in\frac{1}{M}\Z,n>0}a_nq^n\in S_k(\Gamma(M),\chi_1) \text{ et } g=\sum\limits_{n\in\frac{1}{M}\Z, n\geq
0}b_nq^n\in\cM_{k-j}(\Gamma(M),\chi_2)\] des formes propres pour tous
les op\'erateurs $T(l)$ avec $(l,M)=1$. Sous les hypothèses ci-dessus, on a les produits d'Euler suivants:
\begin{equation*}
\begin{split}
\sum\limits_{n>0}\frac{a_n}{n^s}=(\sum\limits_{n\in\Z[\frac{1}{M}]^{*},n>0}\frac{a_n}{n^s})\cdot\prod\limits_{l\nmid M}\frac{1}{(1-\alpha_{l,1}l^{-s})(1-\alpha_{l,2}l^{-s})}, \text{ avec } \alpha_{l,1}\alpha_{l,2}=\chi_1(l)l^{k-1},\\
\sum\limits_{n>0}\frac{b_n}{n^s}=(\sum\limits_{n\in\Z[\frac{1}{M}]^{*},n>0}\frac{b_n}{n^s})\cdot\prod\limits_{l\nmid
M}\frac{1}{(1-\beta_{l,1}l^{-s})(1-\beta_{l,2}l^{-s})}, \text{ avec
} \beta_{l,1}\beta_{l,2}=\chi_2(l)l^{k-j-1}.
\end{split}
\end{equation*}

Soient $D_M(f,g,s)$ la série de Dirichlet définie par
\begin{equation*}
D_M(f,g,s)=L_M(\bar{\chi}_1\chi_2,
j+2(s-k+1))\cdot\sum\limits_{n>0}\frac{\bar{a}_nb_n}{n^s},
\end{equation*}
où $L_M(\chi,s)$ est la fonction L de Dirichlet pour le caracètre $\chi$ modulo $M$ (si $\chi$ est primitive, on note simplement par $L(\chi,s)$), et la s\'erie d'Eisenstein non-holomorphe de poids $j$ de niveau $\Gamma(M)$ définie par
\begin{equation*}
E_{M,s}^{(j)}(\tau)=\sum\limits_{c\equiv d-1\equiv 0 [M]}\frac{1}{(c\tau+d)^j}(\frac{\im\tau}{|c\tau+d|^2})^{s+1-k}.
\end{equation*}
En plus, on a $E_{M,k-1}^{(j)}(\tau)=\frac{1}{M^j}\frac{(-2\pi i)^j}{\Gamma(j)} E_{0,\frac{1}{M}}^{(j)}$.

Soit $ \langle,\rangle$ le produit scalaire de Petersson normalisé sur l'espace des formes modulaires de niveau $\Gamma(M)$
donné par la formule: si $f\in S_k(\Gamma(M))$ et $g\in M_k(\Gamma(M))$, on a 
\[ \langle f,g\rangle:=\frac{1}{[\SL_2(\Z):\Gamma(M)]}\int_{\Gamma(M)\backslash\cH}\bar{f} gy^{k}\frac{dxdy}{y^2}. \]

\begin{prop}\cite[proposition 3.4, corollaire 3.5]{PC1}\label{Rankin}
Sous les hypoth\`ese ci-dessus, on a:
\begin{equation}\label{Ran_1}
D_M(f,g,s)=(\sum\limits_{n\in\Z[\frac{1}{M}]^{*}}\frac{\bar{a}_nb_n}{n^s})\cdot\prod\limits_{l\nmid M}\frac{1}{(1-\frac{\ol{\alpha}_{l,1}\beta_{l,1}}{l^s})(1-\frac{\ol{\alpha}_{l,1}\beta_{l,2}}{l^s})(1-\frac{\ol{\alpha}_{l,2}\beta_{l,1}}{l^s})(1-\frac{\ol{\alpha}_{l,2}\beta_{l,2}}{l^s})},
\end{equation}
\begin{equation}\label{Ran_2}
\frac{\Gamma(s)}{(4\pi)^s}D_M(f,g,s)=\frac{[\SL_2(\Z):\Gamma(M)]}{M}\langle f,gE_{M,s}^{(j)}(\tau)\rangle.
\end{equation}
En conséquence, on a \begin{equation}\label{rankin}
\begin{split}
\frac{\Gamma(k-1)}{(4\pi)^{k-1}}\cdot\frac{\Gamma(j)}{(-2i\pi)^j}\cdot
D_M(f,g, k-1)=\frac{[\SL_2(\Z):\Gamma(M)]}{M^{j+1}}\langle f,gE_{0,\frac{1}{M}}^{(j)}\rangle.
\end{split}
\end{equation}

\end{prop}

\subsubsection{Périodes et algébricité de valeurs spéciales}\label{cal}
 
Soient $\chi_1$ un caractère de Dirichlet modulo $M$ avec $(M,N)=1$ et $\chi_2$ un des caractères $\{\chi,\chi^2\}$ choisi de telle sorte que\footnote{La condition $L(f^*_\alpha, \chi_2^{-1}, k-1)\neq 0$ est automatique si $k>3$. } 
\[\chi_1\chi_2(-1)=(-1)^{k-j} \text{ et }L(f^*_\alpha, \chi_2^{-1}, k-1)\neq 0.\] Si $H=MN$, on pose 
\[F_{\chi_1,\chi_2}^{(k-j)}=\frac{1}{2G(\chi_2)}\sum_{a=1}^{H}\sum_{b=1}^{H}\chi_1(a)\chi_2(b)F^{(k-j)}_{\frac{a}{H}, \frac{b}{H}}=\frac{M}{2G(\chi_2)}\sum_{a=1}^{H}\sum_{b=1}^{N}\chi_1(a)\chi_2(b)F^{(k-j)}_{\frac{a}{H}, \frac{b}{N}}(q^M),\]
où $\chi_1$ est vu comme un caractère modulo $M$ et le dernier égalité est de la relation de distribution de $F_{\alpha,\beta}^{(k)}$ (cf. \cite[lemme 2.6]{Wang}). 

\begin{lemma}\label{q_F}$(1)$ Soit $\gamma=(\begin{smallmatrix}a&b\\c&d\end{smallmatrix})\in\SL_2(\Z)$ avec $b\equiv c\equiv 0\mod H$. Alors l'action de $\gamma$ sur $F_{\chi_1,\chi_2}^{(k-j)}$ est donn\'ee par la formule
\[(F_{\chi_1,\chi_2}^{(k-j)})_{|_{k-j}}\gamma=\chi_1\chi_2^{-1}(d)F_{\chi_1,\chi_2}^{(k-j)}.\]
$(2)$ Soit $\sum\limits_{n\in\Q_+}c_nq^n$ le $q$-d\'eveloppement de
$F_{\chi_1,\chi_2}^{(k-j)}$. On a
\begin{equation*}
\sum\limits_{n\in\Q_{+}^{*}}\frac{c_n}{n^s}=N^{s-1}H^{2-(k-j)}L_{M}(\chi_1,s-(k-j)+1)\cdot
L(\chi_2^{-1},s).
\end{equation*}
\end{lemma}

\begin{proof}[D\'emonstration]
(1) Pour tout $\gamma=(\begin{smallmatrix}a&b\\c&d\end{smallmatrix})\in\SL_2(\Z)$, par la formule (\ref{eis}), on a:
\begin{equation*}
(F_{\chi_1,\chi_2}^{(k-j)})_{|_{k-j}}\gamma=\frac{1}{2G(\chi_2)}\sum\limits_{a_0=1}^{H}\sum\limits_{b_0=1}^{H}\chi_1(a_0)\chi_2(b_0)F^{(k-j)}_{a\frac{a_0}{H}+c\frac{b_0}{H},b\frac{a_0}{H}+d\frac{b_0}{H}}.
\end{equation*}
Si $b\equiv c\equiv 0\mod H$, on a $F^{(k-j)}_{a\frac{a_0}{H}+c\frac{b_0}{H},b\frac{a_0}{H}+d\frac{b_0}{H}}=F^{(k-j)}_{\frac{aa_0}{H},\frac{db_0}{H}}$. On en déduit que
\begin{equation*}
\begin{split}
(F_{\chi_1,\chi_2}^{(k-j)})_{|_{k-j}}\gamma=\frac{1}{2G(\chi_2)}\sum\limits_{a_0=1}^{H}\sum\limits_{b_0=1}^{H}\chi_1(aa_0)\chi_2(db_0)\chi_1\chi_2^{-1}(d)F^{(k-j)}_{\frac{aa_0}{H},\frac{db_0}{H}}=\chi_1\chi_2^{-1}(d)F_{\chi_1,\chi_2}^{(k-j)},
\end{split}
\end{equation*}
où la première \'egalit\'e vient de la relation $\chi_i(\det \gamma)=\chi_i(ad)=1$, pour $i=1,2$.

(2)Soit $\sum\limits_{n\in\Q_+}b_nq^n$ le $q$-d\'eveloppement de
$\frac{M}{2G(\chi_2)}\sum\limits_{a=1}^{H}\sum\limits_{b=1}^{N}\chi_1(a)\chi_2(b)F^{(k-j)}_{\frac{a}{H}, \frac{b}{N}}$. Rappelons que la série de Dirichlet formelle associée à $F_{\frac{a}{H},\frac{b}{N}}^{(k-j)}$ est 
\[\z(\frac{a}{H},s-(k-j)+1)\z^{*}(\frac{b}{N},s)+(-1)^{k-j}\z(-\frac{a}{N},s-(k-j)+1)\z^{*}(-\frac{b}{N},s).\]
Par ailleurs, on a $L_{M}(\chi_1,s)=H^{-s}\sum\limits_{a=1}^{H}\chi_1(a)\z(\frac{a}{H},s)$. 
 Ceci implique, en utilisant la hypothèse $(-1)^{k-j}=\chi_1\chi_2(-1)$, que
\begin{equation*}
\begin{split}
\sum\limits_{n\in\Q_+^{*}}\frac{b_{n}}{n^s}=\frac{M}{2G(\chi_2)}H^{s-(k-j)+1}L_{M}(\chi_1,s-(k-j)+1)\cdot A,
\end{split}
\end{equation*}
où $A=\sum\limits_{n=1}^{+\infty}n^{-s}\sum\limits_{b=1}^{N}\chi_2(b)e^{\frac{2i\pi bn}{N}}
+\sum\limits_{n=1}^{+\infty}n^{-s}\sum\limits_{b=1}^{N} \chi_2(-b) e^{\frac{-2\pi ibn}{N}}$.
Comme $\chi_2$ est primitive, on conclut la proposition de la relation $\sum\limits_{b=1}^{M_2}\chi_2(b)e^{\frac{2\pi
ibn}{M_2}}=\chi_2^{-1}(n)G(\chi_2)$, de $H=NM$ et de l'égalité 
\[\sum\limits_{n\in\Q_+^{*}}\frac{c_n}{n^s}=\sum\limits_{n\in\Q_+^{*}}\frac{b_{n/M}}{n^s} =M^{-s}\sum\limits_{n\in\Q_+^{*}}\frac{b_{n}}{n^s}.\] 

\end{proof}

\begin{prop}\label{vsp}Soient $f=\sum_{n\geq 1}a_nq^n\in S_k(\Gamma_0(N),\eps)$, $\chi_1,\chi_2$ et $H$ ci-dessus.
Si $1\leq j\leq k-1$, alors, 
\begin{equation*}
\frac{[\SL_2(\Z):\Gamma(H)]}{H^{3-k+2j}N^{(k-j)-2}}\langle f_\alpha,F_{\chi_1,\chi_2}^{(k-j)}E_{0,\frac{1}{H}}^{(j)}\rangle
=\frac{\Gamma(k-1)\Gamma(j) }{(4\pi)^{k-1}(-2i\pi)^j} \chi_1(N)L_{M}(f^*_\alpha,\chi_1,j)L(f^*_\alpha,\chi^{-1}_2,k-1),
 \end{equation*}
où $L_{M}(f^*_\alpha,\chi_1,s)$ est la fonction L de $f^*_\alpha$ tordue par le caractère $\chi_1$ modulo $M$.
 
\end{prop}

\begin{proof}De la formule $(\ref{rankin})$, on a 
\begin{equation*}
\begin{split}
&\frac{[\SL_2(\Z):\Gamma(H)]}{H^{j+1}}\langle f_\alpha,F_{\chi_1,\chi_2}^{(k-j)}E_{0,\frac{1}{H}}^{(j)}\rangle=\frac{\Gamma(k-1)}{(4\pi)^{k-1}}\frac{\Gamma(j)}{(-2i\pi)^j}D_H(f_\alpha,F_{\chi_1,\chi_2}^{(k-j)},k-1).
\end{split}
\end{equation*}
On se ramène donc à calculer $D_H(f_\alpha,F_{\chi_1,\chi_2,H}^{(k-j)},s)$. D'après le lemme $\ref{q_F}$, la série de Dirichlet associée à $F_{\chi_1,\chi_2}^{(k-j)}$ est 
\begin{equation*}
\begin{split}
&N^{s-1}H^{2-(k-j)}L_{M}(\chi_1,s-(k-j)+1)\cdot
L(\chi_2^{-1},s)\\
=&M^{2-(k-j)}(\sum_{n\in \Z[\frac{1}{N}]^*, n\in \frac{1}{N}\Z}\frac{\chi_1(Nn)}{n^{s-(k-j)+1}})\cdot\prod_{(\ell,N)=1}\frac{1}{(1-\chi_1(\ell)\ell^{-s+(k-j)-1})(1-\chi_2^{-1}(\ell)\ell^{-s})}.
\end{split}
\end{equation*} 
D'autre part, si $f_\alpha=\sum_{n\in\Q_{+}}b_nq^n$, on a $b_n=0$ si $n\notin\N$, $b_n=a_n$ si $p\nmid n$ et 
\[L(f_\alpha,s)=\frac{1}{1-\alpha p^{-s}}\prod_{\ell\mid N}\frac{1}{1-a_\ell\ell^{-s}}\prod_{\ell\nmid Np}\frac{1}{(1-\alpha_\ell\ell^{-s})(1-\beta_\ell\ell^{-s})}, \text{ avec } \alpha_\ell\beta_\ell=\eps(\ell)\ell^{k-1}.\]
Un calcul direct, en utilisant la méthode de Rankin, montre que 
\begin{equation*}
D_H(f_\alpha,F_{\chi_1,\chi_2,H}^{(k-j)},s)
=M^{2-(k-j)}\chi_1(N)L_{M}(f^*_\alpha,\chi_1,s-(k-j)+1)L(f^*_\alpha,\chi^{-1}_2,s).
\end{equation*}
On conclut la démonstration en prenant $s=k-1$.


\end{proof}

Il est bien connu (\cite{RR},\cite{Sh}) que les nombres
\[\frac{L_M(f^{*}_\alpha, \chi_1, j)L(f^{*}_\alpha,\chi_2^{-1},k-1)}{\pi^{k-1+j}\langle f_\alpha,f_\alpha\rangle},\]
pour $j\in\N$ et $1\leq j\leq k-1$, et $\chi_1\chi_2(-1)=(-1)^{k-j}$, sont algébriques. On pose 
\[ 
\tilde{L}_{\chi_2}(f_\alpha^*,\chi_1, j)=
\Omega(f_\alpha,\chi_2)  \frac{
L_{M}(f^*_\alpha,\chi_1,j)}{(-2i\pi)^j},\]
où $\Omega(f_\alpha,\chi_2)=
\frac{\Gamma(k-1)}{[\SL_2(\Z):\Gamma(N;p)]}\frac{ L(f^*_\alpha,\chi^{-1}_2,k-1)}{(4\pi)^{k-1}\langle f_\alpha,f_\alpha\rangle }$ est une "période", non nulle par hypothèse. Nous allons interpoler les $\tilde{L}_{\chi_2}(f_{\alpha}^*,\chi_1,j)$, pour $1\leq j\leq k-1$ et $\chi_1$ un caractère de conducteur une puissance de $p$, pour construire la fonction L $p$-adique attachée à $f_\alpha$.

\subsection{Raffinement du système d'Euler de Kato}\label{classique} 
Dans ce paragraphe, on effectue la projection du système d'Euler de Kato sur une forme modulaire raffinée $f_\alpha$ en utilisant la méthode de Rankin-Selberg et on obtient un système d'Euler de Kato raffiné associé à $f_\alpha$. De plus, on donne une caractérisation de ce système d'Euler de Kato raffiné via l'exponentielle duale de Bloch-Kato (cf. théorème \ref{Single}).
\subsubsection{Des fonctions localement constantes}
Si $A,M_1,M_2,B$ sont des entiers $\geq 1$, on définit le sous-groupe $\hat{\Gamma}_{A(M_1),B(M_2)}$ comme l'intersection de $\GL_2(\hat{\Z})$ avec l'ensemble des matrices $(\begin{smallmatrix}a&b\\ c& d\end{smallmatrix})$ de $\bM_2(\hat{\Z})$ vérifiant $a-1\in A\hat{\Z}$, $b\in AM_1\hat{\Z}$, $c\in BM_2\hat{\Z}$, $d-1\in B\hat{\Z}$.
En particulier, on note $\hat{\Gamma}_{A,B}$ le groupe $\hat{\Gamma}_{A(M_1),B(M_2)}$ avec $M_1=M_2=1$.

Fixons un caractère de Dirichlet $\xi$ de conducteur $N$. Soit $H$ un multiple de $N$ et $p$ tel que $\cP(H)=\cP(N)\cup\{p\}$, où $\cP(N)$ désigne l'ensemble des facteurs premiers de $N$. Soient $a,b$ deux entiers tels que $1\leq a\leq H$, $(a,p)=1$ et $1\leq b\leq H$. On pose 
\begin{equation*}
\begin{split}A_{a,b,H}=\{(\begin{smallmatrix}\alpha&\beta \\ \delta& \gamma\end{smallmatrix})\in\bM_2(\hat{\Z})^{(p)}: \beta-b\equiv \alpha-a\equiv \delta\equiv \gamma-1\equiv 0\mod H\},\end{split}\end{equation*}
et on définit une fonction localement constante sur $\bM_2(\hat{\Z})^{(p)}$,
\begin{equation*}
\begin{split}
\phi_{\xi,H}=\frac{N}{H}\sum\limits_{a=1, (a,p)=1}^{H}\sum\limits_{b=1}^{H} \xi(b)1_{A_{a,b,H}},
\end{split}\end{equation*}invariante sous l'action de $\hat{\Gamma}(1(N),H)$ à droite (i.e. $\phi_{\xi,H}*\gamma=\phi_{\xi,H}(x\gamma^{-1})$, si $\gamma\in \hat{\Gamma}_{1(N),H}$).

\begin{prop}\label{independant}Si $H',H$ sont deux multiples de $N$ et $p$ tels que $\cP(H)=\cP(H')=\cP(N)\cup\{p\}$, alors on a 
\[\cor_{\hat{\Gamma}(1(H),H)}^{\hat{\Gamma}(1(N),Np)}\phi_{\chi,H}=\cor^{\hat{\Gamma}(1(N),Np)}_{\hat{\Gamma}(1(H'),H')}\phi_{\chi,H'}.\] 
\end{prop}
\begin{proof}Il s'agit à montrer que $\cor_{\hat{\Gamma}(1(H),H)}^{\hat{\Gamma}(1(H'),H')}\phi_{\chi,H}=\phi_{\chi,H'}$ pour $H'|H$. Comme $\phi_{\chi,H}$ est invariante sous l'action de $\hat{\Gamma}(1(N),H)$, alors $\cor_{\hat{\Gamma}(1(H),H)}^{\hat{\Gamma}(1(H'),H)}\phi_{\chi,H}=\frac{H}{H'}\phi_{\chi,H}$.
Posons $m=H/H'$. On note $G$ (resp. $L$) le sous-groupe de $\GL_2(\Z/ H)$ tel que 
\[\hat{\Gamma}(1(H'),H)\backslash\GL_2(\hat{\Z})\cong G\backslash\GL_2(\Z/ H) (\text{resp. } \hat{\Gamma}(1(H'),H')\backslash\GL_2(\hat{\Z})\cong L\backslash\GL_2(\Z/ H)).\] Comme $\cP(H)=\cP(H')$,  pour tout $(x,y)\in(\Z/m)^2$, on peut fixer un élément $s_{x,y}$ de $\GL_2(\Z/H)$ de la forme $(\begin{smallmatrix}1& 0\\H'v &1+H'u\end{smallmatrix})$ tel que $u\equiv x\mod m$ et $v\equiv y\mod m$. Les $s_{x,y}$ forment un système de représentants de $L\backslash G$.  Alors, la fonction localement constante $\frac{H}{H'}\phi_{\xi,H}$ est envoyé sur la fonction localement constante $\sum_{(x,y)\in (\Z/m)^2}\frac{H}{H'}\phi_{\xi,H}*s_{x,y}= \phi_{\xi,H'}$ par l'application de corestriction $\cor_{\hat{\Gamma}(1(H'),H)}^{\hat{\Gamma}(1(H'),H')}$.
  \end{proof}
On pose $\phi_\xi=\cor_{\hat{\Gamma}(1(H),H)}^{K}\phi_{\xi,H}$, qui est indépendante du choix de $H$ par la proposition ci-dessus. En plus, si $\gamma=(\begin{smallmatrix}a&0\\0& d\end{smallmatrix})\in \T(\hat{\Z})\cap K$, on a $\gamma\phi_\xi= \xi(a^{-1})\phi_\xi$ .  On va utiliser cette fonction $\phi_\xi$, où $\xi\in \{\chi,\chi^2\}$, pour projeter le système d'Euler de Kato sur les représentations $V_{f_\alpha}$ et $\mathscr{V}$  respectivement.

\subsubsection{Caractérisation du raffinement du système d'Euler de Kato}

Si $V$ est une représentation de de Rham de $\cG_{\Q_p(\z_M)}$,
le cup-produit avec $\log \chi_{\cycl}\in \rH^1(\cG_{\Q_p(\z_M)},\Q_p)$ fournit un isomorphisme
\[\rH^0(\cG_{\Q_p(\z_M)},\B_{\dR}^+(\bar{\Q}_p)\otimes V)\cong \rH^1(\cG_{\Q_p(\z_M)},\B_{\dR}^+(\bar{\Q}_p)\otimes V),\]
et on définit l'application exponentielle duale de Bloch-Kato $\exp^{*}_{\mathbf{BK}}$ comme l'inverse de cet isomorphisme.


Soit $L(\chi)$ une extenstion finie de $L$ qui contient les valeurs du caractère $\chi$. Si $\xi\in\{\chi,\chi^2\}$, soit $\phi_{\xi}\in \LC_c(\bM_2^{(p)}, \cO_{L(\chi)})^{\tilde{K}}$ la fonction constante construite ci-dessus. En appliquant le projecteur d'Iwasawa $\pi_{\phi,V}$ (pour $\phi=\phi_\xi$ et $V=V_{k,j}$) du $\S 3.4$ à la donnée 
$(\phi_{\xi},V_{k,j}, z_{\Kato,c,d,K}(k,j))$, on obtient un élément 
\[z_{\Kato,c,d,\xi}(k,j)\in \rH^1(\cG_{\Q,S},  \fD_0(\Z_p^*,\rH^1(\G(N;p),V_{k,j}) ).\]  
On note $z_{\Kato,c,d,\xi}(f_\alpha,j)$ la projection de $z_{\Kato,c,d,\xi}(k,j)$ sur  
\[(\rH^1(\G(N;p),V_{k,j})\otimes\Q(f_\alpha))_{\pi_{f_\alpha}}\cong V_{f_\alpha}(1-j).\]

\begin{theo}\label{Single}Il existe un élément $z_{\Kato,c,d,\xi}(f_\alpha)$ de $\rH^1(\cG_\Q, \fD_0(\Z_p^*, V_{f_\alpha}))$, tel que, quels que soient $\ell\in\{0, \cdots, k-2 \}$ et $\eta$ un caractère de Dirichlet modulo $p^m$ vérifiant que $\eta\xi(-1)=(-1)^{k-\ell-1}$, on ait 
 \[\exp^*_{\mathbf{BK}}(\int_{\Z_p^*} \eta(x) x^{-\ell}z_{\Kato,c,d,\xi}(f_\alpha))=A_{c,d,\xi}(f_\alpha,\eta x^{-\ell}) \tilde{L}_\xi(f_\alpha^{*},\eta,\ell+1)f_\alpha, \]
où $A_{c,d,\xi}(f_\alpha,\delta)=2G(\xi)(c_p^2-c_p^{3-k}\delta(c_p^{-1})\xi(c^{-1}))(d_p^2-d_p\delta(d_p^{-1})) N^{k-2}\delta(N)$ si $\delta:\Z_p^*\ra \C_p^*$ est un caractère.

 \end{theo}

 \begin{proof}  La construction de $z_{\Kato,c,d,\xi}(f_\alpha, j)$, pour $j=1$, nous fournit un élément dans $\rH^1(\cG_{\Q,S}, \fD_0(\Z_p^*, V_{f_\alpha}))$, noté par $z_{\Kato,c,d,\xi}(f_\alpha)$. On montre dans la suite qu'il est l'élément que l'on cherche. Il ne reste qu'à calculer l'image de l'intégrale $\int_{\Z_p^*}\eta(x) x^{-\ell}z_{\Kato,c,d,\xi}(f_\alpha)$ sous l'application $\exp^*_{\mathbf{BK}}$. Par construction, cela équivaut à calculer l'image de
 \[\int_{\bM_2(\hat{\Z})^{(p)}}\eta(\det x_p)\phi_{\chi}(x)z_{\mathrm{Kato},c,d}(k,\ell+1)\] sous l'application $\exp^*_{\mathbf{BK}}$ et à calculer sa projection sur la composant correspondant à $f_\alpha$ .

(1) Dans ce cas, Kato \cite[\S10]{KK} (revisité par Colmez \cite[\S 2]{PC1}, Scholl \cite{Sc} et l'auteur \cite{Wang}) a construit une autre application exponentielle duale $\exp^{*}_{\Kato}$ pour calculer cette image. Plus précisément, il a montre que $\exp^*_{\mathbf{BK}}=\exp^{*}_{\Kato}$ dans ce cas et on a la loi de réciprocité explicite de Kato 
 \[\exp^{*}_{\Kato}(z_{\Kato,c,d}(k,\ell))=z_{\mathrm{Eis},c,d}(k,\ell)\] pour $1\leq \ell\leq k-1$. Par la loi de réciprocité explicite de Kato,  on a
 \begin{equation*}
\begin{split}
&\exp_{\Kato}^{*}(\int\eta(\det x_p)\phi_{\xi}(x)z_{\Kato,c,d}(k,\ell+1))
=\int\eta(\det x_p)\phi_{\xi}(x)z_{\mathrm{Eis},c,d}(k,\ell+1)\\
&=(c_p^2-c_p^{3-k+\ell}\eta(c_p^{-1})\xi(c^{-1}))(d^2_p-\eta(d_p^{-1})d_p^{\ell+1})\int\eta(\det x_p)\phi_{\xi}(x)z_{\mathrm{Eis}}(k,\ell+1).
\end{split}
\end{equation*}
Ceci implique que, en revenant à la définition de $F_{\alpha,\beta}^{(k)}$ et $E_{\alpha,\beta}^{(k)}$ (cf. \S \ref{variant}), l'image de 
\[\exp_{\Kato}^{*}(\int_{\hat{\Z}^{(p)}}\eta(x)\phi_\xi z_{\Kato,c,d}(k,\ell+1))=A_{c,d,\xi}(\eta x^{-\ell})\cdot\cor_{\Gamma(M)}^{\Gamma(N;p)}\frac{N}{H}\frac{1}{\ell!}H^{k-4-2\ell}F_{\eta,\xi}^{(k-\ell-1)}E^{(\ell+1)}_{0,\frac{1}{H}},\] 
où $H=Np^m$ et $A_{c,d,\xi}(\eta x^{-\ell})=2
G(\xi)(c_p^2-c_p^{3-k+\ell}\eta(c_p^{-1})\xi(c^{-1}))(d_p^2-d_p^{\ell+1}\eta(d_p^{-1})).$

(2) La projection sur la composant correspondant à $f_\alpha$ nous conduit à calculer le produit scalaire de Petersson de $f_\alpha$ avec le produit
de séries d'Eisenstein ci-dessus; on utilise la méthode de Rankin pour ce faire car $\xi\eta(-1)=(-1)^{k-\ell-1}$. En fait, on a
\begin{equation*}
\begin{split}\langle f_\alpha, \exp^{*}_{\mathbf{BK}}(\int_{\Z_p^*}\eta(x) x_p^{-\ell}z_{\mathrm{Kato},c,d,\xi}(f_\alpha))\rangle
=A_{c,d,\xi}(\eta x^{-\ell})[\Gamma(N;p):\Gamma(H)]H^{k-5-2\ell}\frac{N}{\ell!}\langle f_\alpha,F_{\eta,\xi}^{(k-\ell-1)}E^{(\ell+1)}_{0,\frac{1}{H}}\rangle,
\end{split}
\end{equation*}
et on déduit de la proposition \ref{vsp} que 
 \begin{equation*}
\begin{split}
&[\Gamma(N;p):\Gamma(H)]H^{k-5-2\ell}\langle f_\alpha,F_{\eta,\xi}^{(k-\ell-1)}E^{(\ell+1)}_{0,\frac{1}{H}}\rangle\\
=& \frac{N^{k-\ell-3}\Gamma(\ell+1)\Gamma(k-1)}{(4\pi)^{k-1}(-2i\pi)^{\ell+1}[\SL_2(\Z):\Gamma(N;p)]}\eta(N)L(f_\alpha^{*},\eta,\ell+1)L(f_\alpha^{*},\xi^{-1},k-1)\\
=&N^{k-\ell-3}\eta(N)\Gamma(\ell+1)\tilde{L}_{\xi}(f_\alpha^*,\eta,\ell+1) \langle f_\alpha, f_\alpha\rangle.
\end{split}
\end{equation*}
Ceci implique 
\begin{equation*}
\begin{split}
\frac{\langle f_\alpha, \exp^*_{\mathbf{BK}}(\int_{\Z_p^*}\eta(x)z_{\mathrm{Kato},c,d,\xi}(k,\ell+1)  \rangle}{\langle f_\alpha, f_\alpha\rangle}
=A_{c,d,\xi}(f_\alpha, \eta x^{-\ell}) \tilde{L}_\xi(f_\alpha^{*},\eta,\ell+1),
\end{split}
\end{equation*}
où $A_{c,d,\xi}(f_\alpha, \delta)=2G(\xi)\cdot N^{k-2}\delta(N)\cdot (c^2_p-c_p^{3-k}\delta(c_p^{-1})\xi(c^{-1}))(d_p^2-d_p\delta(d_p^{-1}))$, si $\delta$ est un caractère sur $\Z_p^*$.
  \end{proof}
\begin{remark}\label{eliminerc}
$A_{c,d,\xi}(f_\alpha, \eta)$ est un élément dans $\fD_0(\Z_p^*,\bar{\Q}_p)$ si $\eta$ varies analytiquement sur l'espace des caractères continus sur $\Z_p^*$. On définit une mesure $A_{f_\alpha,c,d,\xi}$ par la formule: si $\eta$ est un caractère continu sur $\Z_p^*$
\[\int_{\Z_p^*}\eta(x)A_{f_\alpha,c,d,\xi}=A_{c,d,\xi}(f_\alpha, \eta^{-1}).\] 
En plus, si on choisit $c\in \hat{\Z}^*$ tel que $\xi(c)$ n'est pas une racine $(p-1)p^{\infty}$-ième de l'unité, elle n'a pas autre zéro que $\eta=x$.
\end{remark}
\subsection{Fonction L $p$-adique d'une forme raffinée $f_\alpha$}\label{unevariable}
Dans ce paragraphe, on construit la fonction L $p$-adique associée à la forme $f_\alpha$ à partir d'un raffinement du système d'Euler de Kato $z_{\Kato,c,d,\xi}(f_\alpha)$. La technique que l'on utilise ici est une variante de celle de Colmez \cite{PC1}. Plus précisément, Colmez utilise la $\varphi$-base du $\varphi$-module filtré admissible $\D_{\cris}(V_{f_\alpha})$ pour décomposer un élément de $\D^{\dag}(V_{f_\alpha})^{\psi=1}$ dans $\D_{\rig}(V_{f_\alpha})\otimes\cR[\frac{1}{t}]$ et on utilise la théorie des représentations triangulines plus adaptée aux familles. 
\subsubsection*{La transformée de Fourier }
Si $\eta\in \LC_c(\Q_p, \bar{\Q})$ est constante modulo $p^n$, on définit sa transformée de Fourier $\hat{\eta}\in\LC_c(\Q_p,\bar{\Q}) $ par la formule $\hat{\eta}(x)= p^{-m}\sum_{y\mod p^m} \eta(y)e^{-2i\pi xy}$, où $m$ est un entier arbitraire $\geq \sup(n, -v_p(x))$, $e^{-2i\pi xy}$ est la racine de l'unité d'ordre une puissance de $p$ en utilisant l'isomorphisme $\Q_p/\Z_p\cong \Z[\frac{1}{p}]/\Z$.

Si $\eta$ est un caractère de Dirichlet, on note $G(\eta)$ la somme de Gauss associé à $\eta$. Si $\eta:\Z_p^*\ra \bar{\Q}$ est un caractère de Dirichlet  de conducteur $p^n$,  
on a
\begin{equation}\label{fourier}\hat{\eta}(x)=\begin{cases} \frac{1}{G(\eta^{-1})}\eta^{-1}(p^nx), &\text{ si } n\geq 1;\\ 1_{\Z_p}(x)-\frac{1}{p}1_{p^{-1}\Z_p}(x), &\text{ si }n=0 .\end{cases}\end{equation}



\subsubsection*{Fonction L $p$-adique d'une forme raffinée $f_\alpha$}
Soit $V$ une repésentation $p$-adique de $\cG_{\Q_p}$ de dimension finie. On a le résultat suivant, dû à Fontaine (cf. \cite[théorème 4.8]{PC1}),  qui décrit la cohomologie d'Iwasawa en utilisant la théorie des $(\varphi,\Gamma)$-modules.
\begin{prop}Soit $V$ une représentation $p$-adique de $\cG_{\Q_p}$ de dimension finie. On a un isomorphisme:
\[\Exp^*: \rH^1_{\Iw}(\Q_p,V)\cong \D^{\dag}(V)^{\psi=1},\]
où $\D^{\dag}(V)$ est le $(\varphi,\Gamma)$-module surconvergent associé à $V$ et $\psi$ est l'inverse à gauche de $\varphi$.
\end{prop}

Le $(\varphi,\Gamma)$-module $\D_{\rig}(V_{f_\alpha})$ sur $\cR$ admet une triangulation 
\[0\varsubsetneq D_\alpha=\D_{\rig}(V_{f_\alpha}(1-k))^{\varphi=\alpha,\Gamma=1}\otimes \cR(\chi_{\cycl}^{k-1})\varsubsetneq \D_{\rig}(V_{f_\alpha}).\] Le $(\varphi,\Gamma)$-module $D_\alpha$ de rang $1$  sur $\cR$ est isomorphe à $\cR(\delta_{\alpha,k})$ avec $\delta_{\alpha,k}\in \Hom_\cont(\Q_p^*, L)$ tel que $\delta_{\alpha,k}(p)=\alpha$ et $(\delta_{\alpha,k})|_{\Z_p^*}=x^{k-1}$. Le produit exterieur $\wedge^2\D_{\rig}(V_{f_\alpha})$ est un $(\varphi,\Gamma)$-module étale de rang $1$ sur $\cR$, qui est isomorphe à $\cR(\delta_{\eps,k})$, où $\delta_{\eps, k}\in \Hom_{\cont}(\Q_p^*,L)$ tel que $\delta_{\eps,k}(p)= \eps(p)$ et $(\delta_{\eps,k})|_{\Z_p^*}=x^{k-1}$.
On note $e_\alpha$ et $e$ respectivement les générateurs de $\cR(\delta_{\alpha,k})$ et $\cR(\delta_{\eps,k})$. 

\begin{prop}\label{exist}Soit $v_p(\alpha)>0$. On note $\D=\D_{\rig}(V_{f_\alpha})$. Si $z\in \D^{\dag}(V_{f_\alpha})^{\psi=1}$ et si $z\wedge e_\alpha= w_\alpha e\in\wedge^2 \D  $ avec $w_\alpha\in\cR$, alors il existe une distribution $\mu_{\alpha}$ d'ordre $v_p(\alpha)$ sur $\Z_p$ à valeurs dans $L$ vérifiant $\psi(\mu_{\alpha})=\alpha^{-1}\mu_{\alpha}$ telle que l'on ait
$w_\alpha=\int_{\Z_p}(1+T)^x\mu_{\alpha}$.
\end{prop}
\begin{proof}On déduit de la condition $z\in \D^{\dag}(V_{f_\alpha})^{\psi=1}$ que $\psi(w_\alpha)=\eps(p)\alpha^{-1}w_\alpha$. L'existence d'une distribution est déduite de \cite[proposition I.11]{PC2},  qui dit qu'une solution dans $\cR$ d'une équation de type $\psi(x)-\alpha^{-1} x\in \cR^+$ où $\alpha\in L$ vérifie $v_p(\alpha)>0$, appartient à $\cR^+$, et donc est la transformation d'Amice d'une distribution. On déduit de \cite[proposition V.3.2]{B1} ou \cite[proposition 4.10]{PC1} que $w_\alpha$ est d'ordre $v_p(\alpha)$. 
\end{proof}
\begin{remark}Le théorème de comparaison de Faltings permet d'identifier $f_\alpha$ à un élément de $\D_{\cris}(V_{f_\alpha}(1-k))$. Comme les poids de Hodge-Tate de $V_{f_\alpha}(1-k)$ sont négatifs, on a $\D_{\cris}(V_{f_\alpha}(1-k))\otimes \cR\subset \D_\rig(V_{f_\alpha}(1-k))$ d'après Berger \cite{B}. En utilisant l'isomorphisme $\D_{\rig}(V_{f_\alpha})\cong \D_{\rig}(V_{f_\alpha}(1-k))$, on identifie $f_\alpha$ comme un élémént de $\D_{\rig}(V_{f_\alpha})$ par abus de notation.
\end{remark}

\begin{lemma}On a  $f_\alpha\wedge e_\alpha\neq 0$ dans $\wedge^2 \D_{\rig}(V_{f_\alpha})$. En particulier,  $t^{1-k}f_\alpha\wedge e_\alpha$ forme une base de $\wedge^2\D_{\rig}(V_{f_\alpha})$.
\end{lemma}
\begin{proof}Soit $\tilde{e}_\beta, \tilde{e}_\beta$ une base du $\varphi$-module filtré $\D_{\cris}(V_f)\cong \D_{\cris}(V_{f_\alpha}(1-k))$ telle que $\varphi(\tilde{e}_\alpha)=\alpha\tilde{e}_\alpha$, $\varphi(\tilde{e}_\beta)=\beta\tilde{e}_\beta$ et $e_\alpha= \tilde{e}_\alpha\otimes \chi_{\cycl}^{k-1}$. D'après \cite[\S 1]{FJ}, le $\varphi$-module filtré $\D_{\cris}(V_{f_\alpha}(1-k))$ admet la description suivante:
\[\Fil^i \D_{\cris}(V_{f_\alpha}(1-k))=\left\{
 \begin{aligned} 0; &\text{ si } i\geq k,
 \\  L(\tilde{e}_\alpha+\delta \tilde{e}_\beta)\cong L f_\alpha; &\text{ si } k-1\geq i\geq 1;\\
  L\tilde{e}_\alpha+L\tilde{e}_\beta; &\text{ si } 0\geq i;
        \end{aligned} \right.\]
où $\delta\in L$ et $\delta=0$ si et seulement si la restriction de $V_{f}$ à $\cG_{\Q_p}$ est scindée. On déduit de la hypothèse $v_p(\alpha)< k-1$ que la restriction de $V_{f}$ à $\cG_{\Q_p}$ n'est pas scindée et donc $\delta\neq 0$. Ceci implique que $f_\alpha\wedge e_\alpha \neq 0$ dans $\wedge^2 \D_{\rig}(V_{f_\alpha})$.

D'autre part, on a $t^{1-k}f_\alpha\in t^{1-k}\Fil^{k-1}\D_{\cris}(V_{f_\alpha}(1-k))= \Fil^0 \D_{\cris}(V_{f_\alpha})\subset \D_{\rig}(V_{f_\alpha}) $.
On en déduit que $t^{1-k}f_\alpha\in \D_{\rig}(V_{f_\alpha})$. En particulier, $\varphi (t^{1-k}f_\alpha\wedge e_\alpha)=\eps(p)(t^{1-k}f_\alpha\wedge e_\alpha)$ et $\gamma(t^{1-k}f_\alpha\wedge e_\alpha)= \chi_{\cycl}^{k-1}(t^{1-k}f_\alpha\wedge e_\alpha)$. 

\end{proof}
  \begin{theo}Si $v_p(\alpha)>0$, il existe une distribution $\mu_{f_\alpha,c,d,\xi}$ d'ordre $v_p(\alpha)$ sur $\Z_p$ vérifiant\\
$(1)$  $\Exp^{*}(z_{\Kato,c,d,\xi}(f_\alpha))\wedge e_\alpha=(\int_{\Z_p}(1+T)^x\mu_{f_\alpha,c,d,\xi}) (t^{1-k}f_\alpha\wedge e_\alpha),$\\
$(2)$ si $0\leq \ell\leq k-2$ et si $\eta$ un caractère de Dirichlet de conducteur $p^n$, tel que $\eta\xi(-1)=(-1)^{k-\ell-1}$, on a 
\[\int_{\Z_p^*}\eta(x) x^{\ell}\mu_{f_\alpha,c,d,\xi}=A_{c,d,\xi}(f_\alpha,\eta^{-1} x^{-\ell})\Gamma(\ell+1)\begin{cases}\frac{G(\eta)}{(\eps^{-1}(p)\alpha)^{n}\eta(-1)}p^{\ell n} \tilde{L}_\xi(f_\alpha^{*},\eta^{-1},\ell+1),&\text{ si } n\geq 1\\
(1-\frac{p^{\ell}\eps(p)}{\alpha})(1-\frac{\eps^{-1}(p)\alpha}{p^{\ell+1}})^{-1}\tilde{L}_\xi(f_\alpha^{*},\ell+1) ,&\text{ si } n=0\end{cases},\]
où le facteur $A_{c,d,\chi}(f_\alpha,\delta)$, où $\delta$ est un caractère sur $\Z_p^*$ à valeurs dans $\bar{\Q}_p^*$, est défini dans le théorème \ref{Single}.
\end{theo}
\begin{proof}On déduit l'existence de la distribution $\mu_{f_\alpha,c,d,\xi}$ vérifiant la condition $(1)$ de la proposition \ref{exist}. En plus, on a $\psi(\mu_{f_\alpha,c,d,\xi})=\eps(p)\alpha^{-1}\mu_{f_\alpha,c,d,\xi}$. Il ne reste que à vérifier la deuxième condition.

Pour $0\leq \ell\leq k-2$, fixons un caract\`ere de Dirichlet $\eta$ de conducteur $p^n$ tel que $\eta\xi(-1)=(-1)^{k-1-\ell}$. Soit $m\geq n$ un entier suffisantment grand et on note $K_m= \Q_p(\z_{p^m})$. On démontre le théorème en comparant les deux expressions de
\[ \left(\tr_{K_m/\Q_p,\eta}p^{-m}\varphi^{-m}(\Exp^{*}(z_{\Kato,c,\xi}(f_\alpha))\right )\wedge e_\alpha\]
dans $L_m[[t]] (t^{1-k}f_\alpha\wedge e_\alpha)$, où $\tr_{K_m/\Q_p,\eta}=\sum_{\gamma\in \Gal(K_m/\Q_p)}\eta(\chi_{\cycl}(\gamma))\gamma$ dont son action sur $L_m[[t]]$  est à travers celle sur $L_m$ (i.e. trivialement sur $t$). 

(I) D'une part, en appliquant la loi de réciprocité explicité de 
Cherbonnier et Colmez  \cite[théorème IV 2.1]{CC3} \`a 
la mesure $z_{\Kato,c,d,\xi}(f_\alpha)$, 
on obtient
\[p^{-m}\varphi^{-m} \Exp^{*}(z_{\Kato,c,d,\xi}(f_\alpha) )
=\sum_{\ell\in\Z}\exp^*_{\mathbf{BK},\ell}(\int_{1+p^m\Z_p}x^{-\ell}z_{\Kato,c,d,\xi}(f_\alpha)))\in L_m[[t]]\otimes \D_{\dR}(V_{f_\alpha}),\]
où $\exp^{*}_{\mathbf{BK},\ell}: \rH^1(\cG_{K_m}, V_{f_\alpha}(-\ell))\ra \D_{\dR}(V_{f_\alpha}(-\ell))=t^{\ell}\D_{\dR}(V_{f_\alpha})=t^{1-k+\ell}\D_{\dR}(V_{f_\alpha}(1-k))$ est l'application exponnentielle de Bloch-Kato pour $V_{f_\alpha}(-\ell)$ . On a alors
\begin{equation*}
\begin{split}
&\left(\tr_{K_m/\Q_p,\eta} p^{-m}\varphi^{-m}(\Exp^{*}(z_{\Kato,c,d,\xi}(f_\alpha) )\right)\wedge e_\alpha \\
=&\sum_{a\in\Z_p^{*}/(1+p^m\Z_p)}\eta(a)\sum\limits_{l\in\Z}\exp^{*}_{\mathbf{BK},\ell}(\int_{a+p^m\Z_p}x^{-\ell}z_{\Kato,c,d,\xi}(f_\alpha))\wedge e_\alpha\\
=&\sum_{\ell\in\Z}\exp^{*}_{\mathbf{BK},\ell}(\int_{\Z_p^{*}}\eta(x)x^{-\ell}z_{\Kato,c,d,\xi}(f_\alpha))\wedge e_\alpha.
\end{split}
\end{equation*}
En particulier, de la caract\'erisation de $z_{\Kato,c,d,\xi}(f_\alpha)$ (cf. th\'eor\`eme \ref{Single}), pour $0\leq \ell\leq k-2$, on a l'égalité, 
\[\exp^{*}_{\mathbf{BK},\ell}(\int_{\Z_p^{*}}\eta(x)x^{-\ell}z_{\Kato,c,d,\xi}(f_\alpha))= A_{c,d,\xi}(f_\alpha,\eta x^{-\ell}) \tilde{L}_\xi(f_\alpha^{*},\eta,\ell+1)f_\alpha t^{1-k+\ell}.\]

Ceci implique que, dans $L_m[[t]] (t^{1-k}f_\alpha\wedge e_\alpha)$, pour $0\leq \ell\leq k-2$, le coefficient en $t^{\ell}(t^{1-k}f_\alpha\wedge e_\alpha)$ de  
\[\left(\tr_{K_m/\Q_p,\eta} p^{-m}\varphi^{-m}(\Exp^{*}(z_{\Kato,c,d,\xi}(f_\alpha) )\right) \wedge e_\alpha\] est
 $A_{c,d,\xi}(f_\alpha,\eta x^{-\ell}) \tilde{L}_\xi(f_\alpha^{*},\eta,\ell+1)$.

(II) D'autre part, en utilisant la définition de $\mu_{f_\alpha,c,d,\xi}$ et la formule 
\[\varphi^{-m}(\int_{\Z_p}(1+T)^x\mu_{f_\alpha,c,d,\xi})=(\int_{\Z_p}\z_{p^m}^xe^{tx/p^m}\mu_{f_\alpha,c,d,\xi}),\] on obtient 
\begin{equation}\label{finally}
\begin{split}
&\left(\tr_{K_m/\Q_p,\eta} p^{-m}\varphi^{-m}(\Exp^{*}(z_{\Kato,c,d,\xi}(f_\alpha) ) \right)\wedge e_\alpha\\
=&\alpha^{m}\tr_{K_m/\Q_p,\eta}  p^{-m}\varphi^{-m}((\int_{\Z_p}(1+T)^x\mu_{f_\alpha,c,d,\xi}) t^{1-k}f_\alpha\wedge e_\alpha)
\\ 
=& \sum_{a\in\Z_p^{*}/1+p^m\Z_p}
\eta(a)p^{-m} (\int_{\Z_p}\z_{p^m}^{ax}e^{\frac{tx}{p^m}}\mu_{f_\alpha,c,d,\xi}) \cdot(\eps^{-1}(p)\alpha)^m(t^{1-k}f_\alpha\wedge e_\alpha).
\end{split}
\end{equation}
On développe $e^{\frac{tx}{p^m}}= \sum_{\ell\geq 0} \frac{1}{\ell!}(\frac{tx}{p^m})^\ell$ et on obtient 
\begin{equation*}
\begin{split}
(\ref{finally})
=\sum_{\ell\geq 0} \frac{t^\ell}{\ell!}\left(\int_{\Z_p}\hat{\eta}(-\frac{x}{p^m}) (\frac{x}{p^m})^\ell\mu_{f_\alpha,c,d,\xi}\right)  (\eps^{-1}(p)\alpha)^m(t^{1-k}f_\alpha\wedge e_\alpha).
\end{split}
\end{equation*}

En comparant les coefficients en $t^\ell$ dans les deux développements ci-dessus, on obtient
\[\frac{(\eps^{-1}(p)\alpha)^m}{\ell!}\int_{\Z_p}\hat{\eta}(-\frac{x}{p^m})(\frac{x}{p^m})^\ell\mu_{f_\alpha,c,d,\xi}= A_{c,d,\xi}(f_\alpha,\eta x^{-\ell}) \tilde{L}_\xi(f_\alpha^{*},\eta,\ell+1).\] 

On déduit, de la formule $(\ref{fourier})$ sur la transformée de Fourier et de la formule $\psi(\mu_{f_\alpha,c,d,\xi})=\eps(p)\alpha^{-1}\mu_{f_\alpha,c,d,\xi}$,  que 
\[\int_{\Z_p}\hat{\eta}(-\frac{x}{p^m})(\frac{x}{p^m})^\ell\mu_{f_\alpha,c,d,\xi}=\begin{cases}\frac{(\eps(p)\alpha^{-1})^{m-n}}{p^{n\ell}G(\eta^{-1})}\int_{\Z_p}\eta^{-1}(-x)x^{\ell}\mu_{f_\alpha,c,d,\xi},&\text{ si } n\geq 1\\ (\eps(p)\alpha^{-1})^{m}(1-\frac{\eps^{-1}(p)\alpha}{p^{\ell+1}})\int_{\Z_p}x^{\ell}\mu_{f_\alpha,c,d,\xi},&\text{ si } n=0.\end{cases}\]




Ceci implique que 
\[\int_{\Z_p^*}\eta^{-1}(x) x^\ell\mu_{f_\alpha,c,d,\xi}=A_{c,d,\xi}(f_\alpha,\eta x^{-\ell})\ell!\cdot \begin{cases} \frac{G(\eta^{-1})\cdot p^{n\ell}}{(\eps^{-1}(p)\alpha)^{n}\eta^{-1}(-1)}\tilde{L}_\xi(f_\alpha^{*},\eta,\ell+1), &\text{ si } n\geq 1\\  (1-\frac{\eps(p)p^\ell}{\alpha})(1-\frac{\eps^{-1}(p)\alpha}{p^{\ell+1}})^{-1}\tilde{L}_\xi(f_\alpha^{*},\ell+1), &\text{ si } n=0\end{cases},\]
où le facteur $(1-\frac{\eps(p)p^\ell}{\alpha})$ dans le cas $n=0$ viens de la formule 
\[\int_{\Z_p^*} x^\ell\mu_{f_\alpha,c,d,\chi}=\int_{\Z_p} x^\ell(1-\varphi\psi)(\mu_{f_\alpha,c,d,\xi}).\]

\end{proof}

Dans la suite, supposons que $\chi$ et $c$ sont bien choisi tels que $\chi(c)$ ne soit pas une racine $(p-1)p^{\infty}$-ième de l'unité. On définit un élément $\mu_{f_\alpha,\chi}$, associée à $f_\alpha$, dans le corps des fractions des distributions sur $\Z_p^*$ à valeurs dans $\bar{\Q}_p$, par la formule: 
\[
\mu_{f_\alpha,\chi}= \frac{\mu_{f_\alpha,c,d,\chi}+(-1)^{k}\delta_{-1}\mu_{f_\alpha,c,d,\chi}}{2A_{f_\alpha,c,d,\chi}}+\frac{\mu_{f_\alpha,c,d,\chi^2}+(-1)^{k-1}\delta_{-1}\mu_{f_\alpha,c,d,\chi^2}}{2A_{f_\alpha,c,d,\chi^2}}, \]
où $\delta_{-1}$ est la masse de Dirac en $-1$, $A_{f_\alpha,c,d,\chi}$ et $A_{f_\alpha,c,d,\chi^2}$ sont les mesures sur $\Z_p^*$ définies dans la remarque \ref{eliminerc}.

\begin{theo}\label{mesureinv}
(1)Si $\eta$ est un caractère continu sur $\Z_p^*$, on a 
\begin{equation*}
\int_{\Z_p^*}\eta\mu_{f_\alpha,\chi}=\begin{cases}\int_{\Z_p^*}\eta\frac{\mu_{f_\alpha,c,d,\chi}}{A_{f_\alpha,c,d,\chi}}, & \text{ si } \eta(-1)=(-1)^{k};\\
\int_{\Z_p^*}\eta\frac{\mu_{f_\alpha,c,d,\chi^2}}{A_{f_\alpha,c,d,\chi^2}},& \text{ si } \eta(-1)=(-1)^{k-1}.\end{cases}
\end{equation*}
(2)
$\mu_{f_\alpha,\chi}$ est une distribution sur $\Z_p^*$, qui est indépendant du choix de $c,d\in \hat{\Z}^*$.
\end{theo}
\begin{proof} Il suffit de montrer que $\mu_{f_\alpha,\chi}$ est une distribution. On rappelle que les mesures $A_{f_\alpha,c,d,\chi}$ et $A_{f_\alpha,c,d,\chi^2}$ ont au plus des zéros en $\eta=x$. Si $k$ est impair (resp. pair), $\eta=x$ est dans le domaine d'intepolation de $\mu_{f_\alpha,c,d,\chi}$ (resp. $\mu_{f_\alpha,c,d,\chi^2}$). On donnera l'argument pour $k$ impair et le cas restant se démontre de la même manière. Le théorème précédent dit que \[\int_{\Z_p^*}x\mu_{f_\alpha,c,d,\chi}
=A_{c,d,\chi}(f_\alpha,x^{-1})(1-\frac{p\eps(p)}{\alpha})(1-\frac{\eps^{-1}(p)\alpha}{p^{2}})^{-1}\tilde{L}_\chi(f_\alpha^{*},2), \]
où $(1-\frac{p\eps(p)}{\alpha})(1-\frac{\eps^{-1}(p)\alpha}{p^{2}})^{-1}\tilde{L}_\chi(f_\alpha^{*},2)$ est un nombre algébrique. Ceci implique que, si le facteur $A_{c,d,\chi}(f_\alpha,x^{-1})=0$, alors $\int_{\Z_p^*}x\mu_{f_\alpha,c,d,\chi}=0$ et donc $\mu_{f_\alpha,\chi}$ est une distribution. 
\end{proof}
\subsection{Fonction L $p$-adique en deux variables}\label{fonct_final}

Soit $S$ une $\Q_p$-algèbre affinoïde. On définit les anneaux $\cE^+_S, \cR^+_S$ par les formules:
\[\cE^+_S=\cE^+\hat{\otimes}_{\Q_p}S \text{ et } \cR^+_S=\cR^+\hat{\otimes}_{\Q_p}S.\]
Si $\Gamma=\Gal(\Q_p(\z_{p^{\infty}})/\Q_p)$, on définit de la même manière
\[\cE^+_S(\Gamma)=\cE^+(\Gamma)\hat{\otimes}S \text{ et } \cR_S^+(\Gamma)= \cR^+(\Gamma)\hat{\otimes}S.\]

Soit $\sX$ un espace rigide analytique réduit et séparé. Soit $\sU$ une famille de représentations $p$-adiques sur $\sX$. Supposons que $\sU$ admet une structure entière, c'est à dire, pour tout $X=\Spm S$ ouvert affinoïde de $\sX$,  la $S$-représentation $\sU_S$ contient un sous $S^+$-module $\sU_S^+$ invariant sous l'action de $\cG_{\Q_p}$, tel que $\sU_S^+\otimes\Q_p=\sU_S$, où $S^+$ la $\Z_p$-algèbre des éléments $s\in S$ de norme $\leq 1$ pour la norme de Banach sur $S$.

On note $\rH_{\Iw}^{i}(\Q_p, \sU)$ le faisceau\footnote{On déduit que ce préfaisceau est un faisceau de la proposition $\ref{grandexp}$.} sur $\sX$ défini comme suit: pour tout $X=\Spm S$ ouvert affinoïde de $\sX$, on pose 
\[\rH_{\Iw}^{i}(\Q_p, \sU)(X)= \rH^i(\cG_{\Q_p}, \fD_0(\Z_p^*,\sU_S)).\]
 Fixons un ouvert affinoïde $X=\Spm S$.  Le module $\rH_{\Iw}^{i}(\Q_p, \sU)(X)$ est un $\cE_S^+(\Gamma)$-module.

Le résultat suivant, dû à Kedlaya-Pottharst-Xiao (cf. \cite[corollary 4.4.11]{KJX}), est une version en famille de l'application exponentielle duale $\Exp^*$ de Fontaine:
\begin{prop}\label{grandexp}Soit $V_S$ une $S$-représentation localement libre de $\cG_{\Q_p}$ muni d'une structure entière. Alors on a un isomorphisme de $\cR_S^+(\Gamma)$-modules:
\[\Exp^*_{S}: \rH^1_{\Iw}(\Q_p, V_S)\hat{\otimes}_{\cE^+(\Gamma)}\cR^+(\Gamma)\cong \D_{\mathrm{BC},\rig}(V_S)^{\psi=1},\]
où $\psi$ est l'inverse à gauche de $\varphi$. Cet isomorphisme est compatible avec le changement de base. En particulier, 
si $x\in X=\Spm S$ , on a 
\[\mathrm{Ev}_x(\Exp^*_S)=\Exp^*: \rH^1_{\Iw}(\Q_p, V_x)\otimes_{\cE^+(\Gamma)}\cR^+(\Gamma)\cong \D_{\rig}(V_x)^{\psi=1}.\]  
\end{prop}

Le faisceau $\sV$ (cf. \S 4.2) n'est pas une vraie famille de représentations galoisiennes sur $\fC^0$ et on utilise la technique "la transformation stricte" de Bellaïche-Chenevier \cite{BeCh}, rappelée ci-dessous, pour le modifier en une vraie famille de représentations de $\cG_{\Q}$ sur $\tilde{\fC}^0$ la normalisation de $\fC^0$.

\begin{defn}\label{strict}Un morphisme $\pi: \sX'\ra \sX $ d'espaces rigides réduits est birationnel s'il existe un faisceau d'idéal cohérent $H\subset \cO_\sX$, tel que,
le complément $U$ du sous-espace fermé $V(H)$ défini par $H$ est Zariski-dense dans $\sX$, le morphisme $\pi$ induit un isomorphisme $\pi^{-1}(U)\ra U$ et l'image inverse $\pi^{-1}(U)$ est Zariski-dense dans $\sX'$.
\end{defn}
Soit $\pi:\sX'\ra \sX$ un morphisme propre et rationnel d'espaces rigides réduits. Fixons un faisceau d'idéal cohérent $H$ dans la définition ci-dessus. Si $\cM$ est un $\cO_\sX$-faisceau cohérent, on définit un $\cO_{\sX'}$-faisceau cohérent $\cM'$, appelé la transformée stricte $\cM'$ de $\cM$, en quotientant le $\cO_{\sX'}$-faisceau cohérent $\pi^*\cM$ par ses $H^{'\infty}$-torsions, où $H'$ est le faisceau d'idéal cohérent définissant le sous-ensemble fermé $\pi^{-1}(V(H))\subset\sX'$.  En général, la définition de la transformée stricte de $\cM$ dépend du choix de $H$. Si $\cM$ est muni d'une action continue $\cO_\sX$-linéaire d'un groupe topologie $G$, ceci induit une action continue $\cO_{\sX'}$-linéaire sur $\cM'$ et l'application naturelle $\pi^*\cM\ra \cM'$ est $G$-équivariant. En particulier, si $\cM$ est sans torsion, alors $\cM'$ est aussi sans torsion et $\cM'$ ne dépend pas du choix de $H$.

\begin{lemma} \cite[lemma 3.4.2]{BeCh} Soit $\cM$ un faisceau cohérent sans torsion sur un espace rigide réduit $\sX$. Il existe un morphisme propre et birationnel $\pi: \sX'\ra\sX$ avec $\sX'$ reduit, tel que, la transformée stricte $\cM'$ de $\cM$ par $\pi$ est un faisceau cohérent localement libre. En particulier, si $\sX$ est une courbe, on peut prendre $\pi$ la normalisation de la courbe. 
\end{lemma}

Soit $\sV'$ la transformée stricte de $\sV$ par le morphisme propre et birationnel $\pi: \tilde{\fC}^0\ra \fC^0$. En particulier, $\sV'$ est une famille de représentations de $\cG_{\Q_p}$.  On note $Z$ le sous-ensemble de $\fC^0$ des formes cuspidales raffinées régulières non-critiques, qui est Zariski-dense dans $\tilde{\fC}^0$. On note $\tilde{\alpha}\in \cO(\tilde{\fC}^0)$ l'image inverse de $\alpha\in \cO(\fC^0)$ sous l'application de normalisation. 
\begin{prop}\label{faible} La donnée $(\sV', \tilde{\alpha}, Z,\kappa_1=\kappa,\kappa_2=0)$ est une famille faiblement raffinée de représentations $p$-adiques de dimension $2$ sur $\tilde{\fC}^0$.
\end{prop}
\begin{proof} Par construction, il suffit de montre que $\sV'$ est une famille de représentation de dimension $2$. Pour tout $x\in \tilde{\fC}^0$, il exist un voisinage $U$ de $x$ tel que l'application de poids est étale sauf en $x$. Comme le rang de la localisation d'un module projective de type fini est localement constant, cela permet de conclure.
\end{proof}

D'après le théorème \ref{triangulation},   $\D_{\mathrm{BC},\rig}(\sV')$ admet un sous-faisceau cohérent localement libre de rang $1$
\[\D_{\tilde{\alpha}}:=\cR\hat{\otimes}\D_{\mathrm{BC},\rig}(\sV')^{\varphi=\tilde{\alpha}, \Gamma=\kappa\circ\chi_{\cycl}}.\]

Pour tout $x\in Z$, il existe un ouvert affinoïde $x\in X$ dans une composante irréductible de $\tilde{\fC}^0$, tel que, les faisceaux cohérents
$\D_{\tilde{\alpha}}$ et $\D_{\mathrm{BC},\rig}(\sV')$ soient libres sur $X$. Dans la suite, par abus de notation, on note $\D_{\mathrm{BC},\rig}(\sV')$ et $\D_{\tilde{\alpha}}$ ses restrictions à $X$.

Soient $e_{\tilde{\alpha}}$ et $e$ deux bases respectivement de $\D_{\tilde{\alpha}}$ et $\wedge^2\D_{\mathrm{BC},\rig}(\sV')$ 
comme $\cR_{X}$-modules.
On a 
\[\varphi(e_{\tilde{\alpha}})=\tilde{\alpha} e_{\tilde{\alpha}}, \gamma (e_{\tilde{\alpha}})= \kappa\circ\chi_{\cycl}(\gamma) ; \varphi(e)=\eps(p)e ,\gamma (e)=\kappa\circ\chi_{\cycl}(\gamma).\]
\begin{prop}\label{mach}Soit $|\tilde{\alpha}|_p\leq 1$ et $\tilde{\alpha}\neq 1$. Soit $z$ une section globale de $\D_{\mathrm{BC},\rig}(\sV')^{\psi=1}$. Alors il existe une section globale $w_F$ de $\cR^+_{X}$ telle que 
$z\wedge e_{\tilde{\alpha}}= w_{\tilde{\alpha}} e$, qui est la transformée d'Amice d'une famille de distributions sur $X$.
\end{prop}
\begin{proof}On a $\psi(z\wedge e_{\tilde{\alpha}})=z\wedge \psi(e_{\tilde{\alpha}})=\tilde{\alpha}^{-1}z\wedge e_{\tilde{\alpha}}=\tilde{\alpha}^{-1}w_{\tilde{\alpha}} e$. D'autre part, on a $\psi(w_{\tilde{\alpha}} e)= \psi(w_{\tilde{\alpha}}) \eps^{-1}(p)e$. On en déduit que 
$\psi(w_{\tilde{\alpha}})= \eps(p)\tilde{\alpha}^{-1} w_{\tilde{\alpha}}$. L'existence de la famille de distributions vient de ce qu'une solution dans $\cR_{\tilde{\fC}^0}$ d'une équation du type $\psi(x)-\alpha x\in\cR_{X}^+ $, où $\alpha\in \cO(X)$ vérifiant $|\alpha|_p\geq 1$ et $\alpha\neq 1$, appartient à $\cR_{X}^+$, et donc est la transformée d'Amice d'une famille de distributions.
\end{proof}
En particulier, cette proposition s'applique à la famille de systèmes d'Euler de Kato $z_{\Kato,c,d,\xi}(X)$, la restriction à $X$ de l'image de $z_{\Kato,c,d,\xi}(\fC^0)$ sous la transformée stricte par rapport à l'application de normalisation. On note $\mu_{X,c,d,\xi}$ la famille de distributions telle que 
\[\Exp^*_{X}(z_{\Kato,c,d,\xi}(X))\wedge e_{\tilde{\alpha}}= \int_{\Z_p}(1+T)^z\mu_{X,c,d,\xi} e.\]

\begin{prop}\label{principal_vrai}Si $f\in Z\cap X$, il existe une constante non-nulle $C(f)$ dépendant de $f$, telle que, on ait 
\[\mathrm{Ev}_f(\mu_{X,c,d,\xi})= C(f)\mu_{f,c,d,\xi}, \text{ où } \mathrm{Ev}_f \text{ est l'application d'évaluation en } f.\]
\end{prop}   
\begin{proof}Si $f_\alpha\in  Z\cap X$ est une forme de poids $k$, on a $t^{1-k}f_\alpha\wedge e_\alpha$ est une base du $(\varphi,\Gamma)$-module $\wedge^2\D_\rig(V_{f_\alpha})$. On a donc $\mathrm{Ev}_{f_\alpha}(e)= C(f_\alpha) t^{1-k} f_\alpha\wedge e_\alpha$, où $C(f_\alpha)\in \bar{\Q}_p^*$ est une constante dépendante de la forme $f_\alpha$.
On déduit de la définition de $\mu_{X,c,d,\xi}$ et $\mu_{f_\alpha,c,d,\xi}$ que 
\[\mathrm{Ev}_{f_\alpha}(\mu_{X,c,d,\xi})= C(f_\alpha)\mu_{f_\alpha,c,d,\xi}.\]
\end{proof}
On note $\kappa:X\ra \sW $ l'application de poids. 
Si $x\in X$,  on note\footnote{On rappelle que l'on a normalisé l'application de poids de telle sorte que $\kappa(f)=k-2$, si $f$ est une forme classique de poids $k$.} $\kappa_x$ le poids de $x$. 
On définit un élément $A_{X,c,d,\xi}$ dans $\fD_0(\Z_p^*, \cO(X))$ par la formule: si $\eta$ est un caractère continu sur $\Z_p^*$ à valeurs dans $\bar{\Q}_p$, on a
 \begin{equation}\label{CD}
\int_{\Z_p^*}\eta A_{X,c,d,\xi}=2G(\xi)\cdot \kappa^{\univ}(N)\delta(N)\cdot (c_p^2-\kappa^{\univ}(c_p^{-1}) c_p\eta(c_p^{-1})\xi(c^{-1}))(d_p^2-d_p\eta(d_p^{-1})). \end{equation}
En plus, si $f_\alpha\in X$ est une forme raffinée comme dans le chapitre précédent, on a 
\[\mathrm{Ev}_{f_\alpha}(A_{X,c,d,\xi})=\mathrm{Ev}_{\kappa_{f_\alpha}}(A_{X,c,d,\xi})= A_{f_\alpha,c,d,\xi}.\] 
Rappelons que $\chi$ et $c$ sont bien choisi tels que $\chi(c)$ ne soit pas une racine $(p-1)p^{\infty}$-ième de l'unité. La mesure $A_{X,c,d,\xi}$ a au plus des zéros sur  $X\times\{x^{-1}\}$. On définit un élément $\mu_{X, \chi}$ dans le corps des fractions de $\fD(\Z_p^*,\cO(X))$ par la formule:
 \[
\mu_{X,\chi}= \frac{\mu_{X,c,d,\chi}+\kappa^{\univ}(-1)\delta_{-1}\mu_{X,c,d,\chi}}{2A_{X,c,d,\chi}}+\frac{\mu_{X,c,d,\chi^2}-\kappa^{\univ}(-1)\delta_{-1}\mu_{X,c,d,\chi^2}}{2A_{X,c,d,\chi^2}}, \]
\begin{theo}(1) Si $f\in Z\cap X$, alors on a $\mathrm{Ev}_f(\mu_{X,\chi})=C(f) \mu_{f,\chi}$, 
où $C(f)$ est une constante non nulle dans $\bar{\Q}_p^*$ dépendant de $f$.\\
(2) $\mu_{X, \chi}$ est une distribution sur $\Z_p^*$ à valeurs dans $\cO(X)$, indépendant du choix de $c,d$. 
\end{theo}
\begin{proof}La propriété d'interpolation est une conséquence du théorème précédent. Il ne reste que à justifier que $\mu_{X,\chi}$ est une distribution. Rappelons que $A_{X,c,d,\chi}$ et $A_{X,c,d,\chi^2}$ n'ont pas d'autre zéros que $X\times\{x\}$. Si $f\in Z\cap X$ , on a 
\[\mathrm{Ev}_f(\int_{\Z_p^*}x\mu_{X,\chi})=\begin{cases}\mathrm{Ev}_f(\frac{\int_{\Z^*_p}x\mu_{X,c,d,\chi}}{\int_{\Z_p^*}xA_{X,c,d,\chi}}), &\text{ si le poids de $f$ est impair};\\
\mathrm{Ev}_f(\frac{\int_{\Z^*_p}x\mu_{X,c,d,\chi^2}}{\int_{\Z_p^*}xA_{X,c,d,\chi^2}}), &\text{ si le poids de $f$ est pair}.\end{cases} \]

La propriété d'interpolation de $\mu_{X,c,d,\chi}$ et $\mu_{X,c,d,\chi^2}$, le théorème \ref{mesureinv} et la densité de $Z\cap X$ dans $X$ nous permet de conclure que l'intégration $\int_{\Z_p^*}x\mu_{X,\chi}$ n'a pas de pôle sur $X$.
\end{proof}

\begin{remark}
 (1) Soit $U\subset\tilde{\fC}^0$ un voisinage d'un point classique non-critique $f_\alpha$ où la fonction L $p$-adique en deux variables de Panchishkin\footnote{On ignore la différence entre les choix du caractère auxiliaire: dans \cite{AP}, il utilise un caractère auxiliaire modulo $p$, et chez nous, on utilise un caractère auxiliaire modulo $N$.} $L_{\mathrm{Pan},\chi}(f,\sigma)$ est définie. Le prolongement analytique nous permet de  déduire qu'il existe un élément $F$ dans le corps des fractions de $\cO(U)$ sans zéro ni pôle sur $f\in Z$, telles que,   
\[
L_{\mathrm{Pan},\chi}(f, \sigma)=F(f)\cdot L_{p,\chi}(f,\sigma)\cdot\mathrm{Eul}_N(f,\sigma) \] 
où $\mathrm{Eul}_N(f,\sigma)$ est un produit de "facteur d'Euler" explicit (cf. \cite{AP} la formule (0.6)) en les $l\mid N$, qui appartient à $\cO(\tilde{\fC}^0)\hat{\otimes} \Lambda$ de manière évidente. \\

(2) Le prolongement analytique nous permet aussi de montrer qu'il existe une fonction $F$ dans le corps des fractions de $\cO(\tilde{\fC}^0)$, telle que 
$L_{p,\chi}(\tilde{\fC}^0,\delta)= F L_{\mathrm{Bel}}(\tilde{\fC}^0,\delta)$, où $L_{\mathrm{Bel}}(\tilde{\fC}^0,\delta)$ est la fonction L $p$-adique en deux variables de Bellaïche (cf. \cite{Be}). Si $x\in \fC^0$ est un point critique au sens de Bellaïche \cite{Be}, on définit la fonction L $p$-adique critque $L_{p,\chi}(x,\delta)$ en $x$ en évaluant $L_{p,\chi}(\tilde{\fC}^0, \delta)$ en $x$.  Si $x\in\fC^0$ est le raffinement critique d'une forme de type CM de poids $k$ et si $\phi:\Z_p^*\ra \Q_p^*$ est un caractère d'ordre fini, alors on a 
\[L_{p,\chi}(x,\phi \cdot x^j)=0, \text{ si } 0\leq j\leq k-2.\]
En effet, on déduit de \cite[proposition 4.5.2]{Liu} que le $\cR$-module $\mathrm{Ev}_x(\D_\alpha)$ n'est pas saturé dans $\D_{\rig}(V_x)$ et le module $t^{1-k}\mathrm{Ev}_x(\D_\alpha)$ est saturé. Ceci implique que $\mu_{x,\chi}=\Ev_x(\mu_{X,\chi})$ est la dérivée $(k-1)$-ième d'une distribution sur $\Z_p$ et que $\int_{\Z_p^*}\phi x^{j}\mu_{x,\chi}=0$ si $0\leq j\leq k-2$ . Il faudrait travailler plus pour verifier que $L_{p,\chi}(x,\sigma)$ n'est pas identiquement nulle et la relier  
à la fonction L critique de Bellaïche $\cite{Be}$.
\end{remark}   

\printindex

\end{document}